\documentclass[11pt]{amsart}

\usepackage{amsmath}
\usepackage{mathtools}
\usepackage{amsthm}
\usepackage{amsfonts}  
\usepackage{amssymb}
\usepackage[margin=1.25in]{geometry}
\usepackage[english]{babel}
\usepackage{hyperref}
\usepackage[shortlabels]{enumitem}
\usepackage{pdfsync}
\usepackage[dvipsnames]{xcolor}
\usepackage{csquotes}
\usepackage[backend=biber,style=alphabetic,doi=false,isbn=false,url=false]{biblatex} 
\renewbibmacro{in:}{}
\usepackage[capitalize,noabbrev]{cleveref}
\usepackage{thmtools}
 
\newtheorem{theorem}{Theorem}[section]
\newtheorem*{theorem*}{Theorem}
\newtheorem{lemma}[theorem]{Lemma}
\newtheorem{proposition}[theorem]{Proposition}
\newtheorem*{proposition*}{Proposition}

\newtheorem{corollary}[theorem]{Corollary}
\theoremstyle{definition}
\newtheorem{example}[theorem]{Example}

\newtheorem{definition}[theorem]{Definition}
\newtheorem{remark}[theorem]{Remark}
\newtheorem*{remark*}{Remark}

\numberwithin{equation}{section}

\newcommand{\R}{\mathbb{R}}
\newcommand{\N}{\mathbb{N}}
\newcommand{\Z}{\mathbb{Z}}

\newcommand{\cat}{\mathbf}
\newcommand{\isom}{\cong}

\newcommand{\supp}{\textup{supp}}

\newcommand{\xto}{\xrightarrow}
\newcommand{\tensor}{\otimes}
\newcommand{\isomto}{\xto{\isom}}

\newcommand{\lipnorm}[1]{L(#1)}
\newcommand{\eps}{\varepsilon}
\newcommand{\pBorel}{\mathcal{B}^+(X,A)}

\newcommand{\pBorelXX}{\mathcal{B}^+(X^2,A^2)}

\newcommand{\pRadon}{\hat{\mathcal{M}}^+_1(X,A)}

\newcommand{\pfRadon}{\mathcal{M}_1^+(X,A)}

\newcommand{\pfRadonA}{\mathcal{M}_1^+(X,d,A)}

\newcommand{\pfRadonXXAA}{\mathcal{M}_1^+(X^2,A^2)}

\newcommand{\Radon}{\hat{\mathcal{M}}_1(X,A)}

\newcommand{\fRadon}{\mathcal{M}_1(X,A)}

\newcommand{\pBorelY}{\mathcal{B}^+(Y,B)}
\newcommand{\pBorelXY}{\mathcal{B}^+(X \times Y,A \times B)}
\newcommand{\pBorelXonly}{\mathcal{B}^+(X)}
\newcommand{\pBorelXXonly}{\mathcal{B}^+(X^2)}
\newcommand{\pBorelAonly}{\mathcal{B}^+(A)}
\newcommand{\lpfpRadon}{\mathcal{\hat{M}}_p^+(X,A)}
\newcommand{\pfpRadon}{\mathcal{M}_p^+(X,A)}
\newcommand{\fpRadon}{\mathcal{M}_0^+(X,A)}
\newcommand{\lpfqRadon}{\mathcal{\hat{M}}_q^+(X,A)}
\newcommand{\pfqRadon}{\mathcal{M}_q^+(X,A)}
\newcommand{\lpfrvRadon}{\mathcal{\hat{M}}_p(X,A)}
\newcommand{\pfrvRadon}{\mathcal{M}_p(X,A)}

\DeclareMathOperator{\op}{op}
\DeclareMathOperator{\Lip}{Lip}
\DeclareMathOperator{\KR}{KR}

\newcommand{\LipX}{\Lip(X,A)}
\newcommand{\LipcX}{\Lip_c(X,A)}
\newcommand{\LipOneX}{\Lip_1(X,A)}
\newcommand{\LipcOneX}{\Lip_{c,1}(X,A)}
\newcommand{\LipLX}{\Lip_L(X,A)}
\newcommand{\LipcLX}{\Lip_{c,L}(X,A)}
\newcommand{\LipXpos}{\Lip^+(X,A)}
\newcommand{\LipcXpos}{\Lip_c^+(X,A)}
\newcommand{\LipLXpos}{\Lip_L^+(X,A)}
\newcommand{\LipOneXpos}{\Lip_1^+(X,A)}
\newcommand{\LipcLXpos}{\Lip_{c,L}^+(X,A)}
\newcommand{\LipcOneXpos}{\Lip_{c,1}^+(X,A)}
\newcommand{\LipY}{\Lip(Y,B)}
\newcommand{\LipOneXXpos}{\Lip_1^+(X^2,A^2)}

\newcommand{\LipXXAA}{\Lip(X^2,A^2)}
\newcommand{\LipXXAApos}{\Lip^+(X^2,A^2)}

\newcommand{\LipcXXAA}{\Lip_c(X^2,A^2)}

\DeclarePairedDelimiter{\abs}{\lvert}{\rvert}
\DeclarePairedDelimiter{\norm}{\lVert}{\rVert}

\newcommand{\KRnorm}[1]{\norm{#1}_{\KR}}

\newcommand{\join}{\vee}
\newcommand{\meet}{\wedge}

\newlist{thmenum}{enumerate}{1}
\setlist[thmenum, 1]{label=(\arabic*), ref=\thetheorem ~(\arabic*)}

\newcommand\restr[2]{{
\left.\kern-\nulldelimiterspace 
#1 
\vphantom{\big|} 
\right|_{#2} 
}}

\begin{document}
	
\title{Relative Optimal Transport}

\author[P. Bubenik]{Peter Bubenik}
\address{Department of Mathematics, University of Florida,
PO Box 118105, Gainesville, FL 32611-8105, USA}
\email{peter.bubenik@ufl.edu}

\author[A. Elchesen]{Alex Elchesen}
\address{Google, 1600 Amphitheatre Parkway, Mountain View, CA 94043, USA}
\email{alexelchesen@gmail.com}

\begin{abstract}
  We develop a theory of optimal transport relative to a distinguished subset, which acts as a reservoir of mass, allowing us to compare measures of different total variation. This relative transportation problem has an optimal solution and we obtain relative versions of the Kantorovich-Rubinstein norm, Wasserstein distance, Kantorovich-Rubinstein duality and Monge-Kantorovich duality.
  We also prove relative versions of the Riesz-Markov-Kakutani theorem, which connect the spaces of measures arising from the relative optimal transport problem to spaces of Lipschitz functions. 
  For a boundedly compact Polish space, we show that our relative 1-finite real-valued Radon measures with relative Kantorovich-Rubinstein norm coincide with the sequentially order continuous dual of relative Lipschitz functions with the operator norm. As part of our work we develop a theory of Riesz cones that may be of independent interest.
\end{abstract}

\subjclass{Primary: 49Q22, 28C05; Secondary: 51F99, 46A40, 55N31}

\maketitle

\tableofcontents
 
\section{Introduction}

A standard setting for optimal transport consists of a metric space $(X,d)$ 
together with two finite measures $\mu, \nu$ on $X$ with $\mu(X) = \nu(X)$.
Following work of Figalli and Gigli~\cite{figalli2010new} and Divol and Lacombe~\cite{divol2019understanding},
we consider a metric space $(X,d)$ together with a distinguished subset $A \subset X$ 
(we call $(X,d,A)$ a metric pair)
and two measures $\mu$ and $\nu$ on $X$.
We seek an optimal transportation plan from $\mu$ to $\nu$ 
relative to $A$, which acts as a reservoir to which we can transport mass or from which we may borrow mass.
Unlike the classical setting, we neither require that $\mu(X) = \nu(X)$, 
nor do we require  that $\mu$ and $\nu$ be finite.
In fact, it will be important to allow measures that need not be locally finite.
Furthermore, we will consider differences of measures, $\mu^+ - \mu^-$, 
for which there may exist disjoint Borel sets $U$ and $V$ with 
both $\mu^+(U) = \infty$ and $\mu^-(V) = \infty$.

Previous work focused on the on the case $(X,d,A)$ where $\Omega = X \setminus A$ was a bounded domain in Euclidean space and $A = \partial \Omega$~\cite{figalli2010new}, and the case where $\Omega$ may be unbounded and $d$ is given by the $q$-norm for $1 \leq q \leq \infty$~\cite{divol2019understanding}. 
However, many application of classical optimal transport extend beyond the Euclidean setting. For example, optimal transport has been used to compare word embeddings \cite{kusnerb2015wordembeddings} and distributions of gene expression profiles \cite{schiebinger2019optimal-transport}. In order to apply optimal transport relative to a subset in such settings, it is therefore necessary to develop the theory for general metric pairs.
We are led to consider both the $p$-finite relative Radon measures and the locally $p$-finite relative Radon measures, which may be of independent interest.
On of our main results, a relative version of Kantorovich-Rubinstein duality \cref{thm:main}, has been stated without proof in the Euclidean case~\cite[Remark 5.3]{Divol:2019b}.

\subsection*{Summary of our results}

Given a metric pair $(X,d,A)$, we define relative Borel measures to be elements of the quotient monoid $\mathcal{B}^+(X)/\mathcal{B}^+(A)$, where $\mathcal{B}^+(X)$ denotes the commutative monoid of Borel measures on $X$ (\cref{sec:borel-pointed}).
There is an obvious bijection between this set of relative Borel measures and the Borel measures on $X \setminus A$.
For relative Borel measures, tightness, that is, inner regularity with respect to compact sets, is well defined.

To facilitate the generalization from possibly infinite measures to differences of such measures, we develop a theory of Riesz cones, analogs of Riesz spaces that are equipped with an $\R^+$-action rather than an $\R$-action (\cref{sec:riesz-cone}). Taking the Grothendieck group of a Riesz cone produces a Riesz space.

We say that a relative Borel measure $\mu$ is $p$-finite if it has finite $p$-th  moment about $A$, that is, $\mu(d_A^p) < \infty$, where $d_A$ is the function on $X$ that gives the distance to $A$ (\cref{sec:finiteness-conditions}).
Given $p \leq q$, 
for measures with support at least some distance away from $A$, 
$q$-finite implies $p$-finite, 
and
for measures with support within some distance of $A$, 
$p$-finite implies $q$-finite.

We define the set of
$p$-finite relative Radon measures, $\pfpRadon$, to be the $p$-finite, tight, relative Borel measures (\cref{sec:weak-radon}).
Then there is an obvious bijection between $\pfpRadon$ and the set of $p$-finite Radon measures on $X \setminus A$. We also define
the set of locally $p$-finite relative Radon measures, $\lpfpRadon$, to be the locally $p$-finite, tight, relative Borel measures.
These
do not seem to have been previously considered. 
These arise naturally (in the case $p=1$) as a dual space to the space of compactly supported Lipschitz functions on $(X,d,A)$ (\cref{thm:1.2}).
They restrict to Radon measures on $X \setminus A$, but they
are not Radon measures on $X$, since they need not be locally finite.
As in the classical case, for $p \neq q$, $\pfpRadon \neq \pfqRadon$.
Unlike the classical case, we also have $\lpfpRadon \neq \lpfqRadon$.
We prove that $\pfpRadon$ and $\lpfpRadon$ are Riesz cones and that $\pfpRadon$ is an ideal in $\lpfpRadon$.
We define $\pfrvRadon$ and $\lpfrvRadon$ to be the Riesz spaces corresponding to $\pfpRadon$ and $\lpfpRadon$ respectively,
and call their elements (locally) $p$-finite real-valued relative Radon measures (\cref{sec:signed}).
$\pfrvRadon$ is an ideal in Riesz space $\lpfrvRadon$.

Consider the Riesz space, $\LipX$, of real-valued Lipschitz functions on $X$ that vanish on $A$ and its ideal $\LipcX$ of compactly supported functions.
$\LipX$ is a Banach space with norm given by the Lipschitz number, $L(-)$, but it is not a Banach lattice or a normed Riesz space.
We show that
the $1$-finite relative Radon measures and the locally $1$-finite Radon measures are the tight relative Borel measures $\mu$ such that $\mu(f) < \infty$ for all $f \in \LipX$ and for all $f \in \LipcX$, respectively (\cref{sec:1-finite}).
In fact, they are positive linear functional functionals on $\LipX$ and $\LipcX$ respectively.
By the monotone convergence theorem,
they are sequentially order continuous.
Furthermore, elements of $\pfRadon$ are exhausted by compact sets (\cref{def:exhausted}).
The following representation theorems provide converse statements.
They may be viewed as relative versions of the Riesz-Markov-Kakutani representation theorem.

\begin{theorem}[\cref{thm:representation-lipc}] \label{thm:1.1}
    Assume that $X$ is locally compact.
    Let $T$ be a sequentially order continuous positive linear functional on $\LipcX$. 
    Then $T$ is represented by a unique $\mu \in \pRadon$.
\end{theorem}

\begin{theorem}[\cref{thm:representation-lipc-non-positive}] \label{thm:1.2}
    If $X$ is locally compact then $\Radon$ is the sequentially order continuous dual of $\LipcX$.
\end{theorem}

\begin{theorem}[\cref{thm:representation-lip}] \label{thm:1.3}
    Let $T$ be a sequentially order continuous positive linear functional on $\LipX$ that is exhausted by compact sets.
    Then $T$ is represented by a unique $\mu \in \pfRadon$.
\end{theorem}

\begin{theorem}[\cref{cor:representation-lip-non-positive}] \label{thm:1.4}
    If $X \setminus A$ is locally compact and $\sigma$-compact
    then $\fRadon$ is the sequentially order continuous dual of $\LipX$.
\end{theorem} 


For additional variants of these results see \cref{thm:representation-lip-non-positive,thm:representation-lip-nonpositive-exhausted,thm:representation-proper,thm:representation-proper-not-positive}.

Since the sequentially order continuous dual on $\LipX$ is an ideal in the order dual of $\LipX$, and since $\pfRadon$ separates points of $\LipX$, we have the following corollary to \cref{thm:1.4}.

\begin{corollary} \label{cor:1.5}
    If $X \setminus A$ is locally compact and $\sigma$-compact then
    $\LipX$ embeds as a Riesz subspace of the order continuous dual of $\fRadon$ by mapping $f$ to $\hat{f}:\mu \mapsto \mu(f)$.
    Since $\LipcX$ is an ideal in $\LipX$,  this mapping also embeds $\LipcX$ as a Riesz subspace of the order continuous dual of $\fRadon$.
\end{corollary}

Similarly, we have the following corollary to \cref{thm:1.2}.

\begin{corollary} \label{cor:1.6}
    If $X$ is locally compact then
    $\LipcX$ embeds as a Riesz subspace of the order continuous dual of $\Radon$ by mapping $f$ to $\hat{f}:\mu \mapsto \mu(f)$.
\end{corollary}

For the remainder of this section, assume that $(X,d)$ is complete and separable.
For $\mu,\nu \in \pfRadon$, define the set of couplings, $\Pi(\mu,\nu)$, to consist of relative Borel measures in the product metric pair $(X,d,A) \times (X,d,A)$ whose marginals are $\mu$ and $\nu$.
Define the \emph{relative $1$-Wasserstein distance} to be given by
\begin{equation} \label{eq:relative-wasserstein}
  W_1(\mu,\nu) = \inf_{\sigma \in \Pi(\mu,\nu)} 
  \int_{X \times X} \bar{d}(x,y) d\sigma, 
  \quad \text{where} \quad 
  \bar{d}(x,y) = \min(d(x,y), d_A(x) + d_A(y)).
\end{equation}

Say that $X$ is boundedly compact if all closed and bounded subsets are compact.

\begin{theorem}[\cref{thm:metric,cor:optimal_coupling}]
  $W_1$ is a metric on $\pfRadonA$ and
  there is an isometric embedding of the quotient metric space $X/A$ into $\pfRadonA$ given by $x \mapsto \delta_{x}$ if $x \not\in A$ and $A \mapsto 0$.
  If $X$ is boundedly compact then there exists an optimal coupling for \eqref{eq:relative-wasserstein}.
\end{theorem}

For the remainder of this section, assume that $X$ is boundedly compact.
We prove the following relative version of Kantorovich-Rubinstein duality.

\begin{theorem}[\cref{thm:KR-duality}] \label{thm:KR-duality-intro}
    Let $\mu,\nu \in \fRadon$. Then
    \[ W_1(\mu,\nu) =  \sup \Bigl\{ \int_X f d(\mu-\nu) \ | \ f \in \LipX, L(f) \leq 1 \Bigr\}
    .\]
    Hence, viewing $\mu -\nu$ as a linear functional on $\LipX$, we have 
    $W_1(\mu,\nu) = \norm{\mu-\nu}_{\textup{op}}$.
\end{theorem}

We also prove the following relative version of Monge-Kantorovich duality.

\begin{theorem}[\cref{thm:MK-duality}]
    Let $\mu,\nu \in \pfRadon$ and $h \in \LipXXAApos$. Then
    \[ \min_{\pi\in \Pi(\mu,\nu)}\pi(h) = \sup \{ \mu(f) + \nu(g) \ | \ f,g \in \LipX, f(x) + g(y) \leq h(x,y), \forall x,y \in X\}.\]
\end{theorem}

We now strengthen \cref{thm:1.4} as follows.

\begin{theorem}[\cref{thm:main_result}]
    Let $T$ be an element of the sequentially order continuous dual of $\LipX$.
    Then $T$ is represented by $\mu \in \fRadon$.
    Furthermore, $\norm{T}_{\textup{op}} = W_1(\mu^+,\mu^-)$,
\end{theorem}

We also have the following.

\begin{theorem}[\cref{thm:main_result_compact}]
    Let $T$ be an element of the sequentially order continuous dual of $\LipcX$
    such that both $T$ and $\abs{T}$ are bounded.
    Then $T$ is represented by $\mu \in \fRadon$.
    Furthermore, $\norm{T}_{\textup{op}} = W_1(\mu^+,\mu^-)$,
\end{theorem}

We show that the metric $W_1$ gives $\pfRadon$ the structure of a normed convex cone (\cref{sec:convex-cones}). 
From this, we obtain the following.

\begin{proposition}[\cref{prop:KR-norm}]
    $\fRadon$ is a normed vector space with norm $\KRnorm{-}$ given by 
    \begin{equation*}
        \KRnorm{\mu} = W_1(\mu^+,\mu^-),
    \end{equation*}
    which we call the \emph{relative Kantorovich-Rubinstein norm}.
\end{proposition}

We now restate our relative version of Kantorovich-Rubinstein duality (\cref{thm:KR-duality-intro}).
\begin{theorem}[\cref{thm:KR-duality2}] \label{thm:main}
  $\fRadon = \LipX_c^{\sim}$, and for $\mu \in \fRadon$, $\norm{\mu}_{\op} = \norm{\mu}_{\KR}$.
%
That is, $(\fRadon, \KRnorm{-})$ embeds isometrically in $\LipcX'$ and its image is the sequentially order continuous dual.
\end{theorem}

Furthermore, for $p \geq 1$, we define a relative $p$-Wasserstein distance and show that it satisfies the triangle inequality (\cref{sec:p-wasserstein}).

\subsection*{Application to topological data analysis}

We were motivated to undertake this work by problems in topological data analysis (TDA).
In particular, let $X$ be a set of parameters for objects in some abelian category (e.g. indecomposables or projectives in a category of persistence modules) with some metric relevant to an application of interest, and a distinguished subset $A$ of parameters corresponding to trivial or ephemeral objects. 
In the classical case of persistence modules consisting of functors from the poset $\R$ to a category of vector spaces, we have the set $\R_\leq^2 = \{(x,y) \in \R^2 \ | \ x \leq y\}$, which parametrizes interval modules, with some metric $d$, and the subset $\Delta = \{(x,y) \in \R^2 \ | \ x=y\}$.
Invariants of interest consist of (signed) formal sums on the metric pair $(X,d,A)$, which are finitely-supported integer-valued relative Radon measures on $(X,d,A)$.
Taking limits, we obtain (locally) $1$-finite relative Radon measures.
To analyze these measures, we want a good class of continuous linear functionals. These are provided by \cref{cor:1.5,cor:1.6}.
Our work provides a framework for optimization of multiparameter persistence~\cite{Scoccola:2024}.

We remark that for persistence modules arising from stationary point processes (e.g. Poisson, binomial), the persistent Betti numbers are asymptotically normal and the persistence diagrams converge to finite Radon measures~\cite{Yogeshwaran:2017,Trinh:2019,Hiraoka:2018,Divol:2019b,Krebs:2024,Botnan:2024}.
However, for persistence modules arising from almost-surely continuous stochastic processes (e.g. Brownian motion with drift), the persistent Betti numbers for $x < x+\eps$ approach $\infty$ as $\eps \to 0$~\cite{chazal2019density,Divol:2019b,Perez:2023,Baryshnikov:2019,Perez:thesis}.

Other cases where persistent Betti numbers diverge include the energy functional on the free loop space of a closed Riemannian manifold~\cite{Ginzburg:2024} and the Floer complex under iterations of a Hamiltonian diffeomorphism~\cite{Ginzburg:2024a}.

\subsection*{Related work}

The idea of relative optimal transport goes back to at least Cohen-Steiner, Edelsbrunner, Harer, and Mileyko \cite{cohen2007stability,cohen2010lipschitz}
They used ideas from optimal transport to introduce the bottleneck and Wasserstein distances for topological summaries called persistence diagrams.
These distances play fundamental roles in the stability theory of persistent homology.
Figalli and Gigli first introduced and studied the relative transport problem in its own right in the setting of measures defined on bounded subsets of Euclidean space \cite{figalli2010new}. They showed that the gradient flow with respect to the relative $2$-Wasserstein distance of a certain entropy functional on measures gives rise to weak solutions of the heat equation with Dirichlet boundary conditions.
In order to develop a theory of optimal transport that included the bottleneck distance as a special case, Divol and Lacombe extended the relative transport problem of \cite{figalli2010new} to measures defined on unbounded subsets of Euclidean space \cite{divol2019understanding}. 
Expectations of distributions of persistence diagrams, which are not themselves persistence diagrams but rather measures supported on the plane, were studied in \cite{chazal2019density}. These are motivating examples of the relative Radon measures in the present paper. 
A framework for performing learning tasks on spaces of Radon measures equipped with the relative $\infty$-Wasserstein distance was developed in \cite{elchesen2022learning}.
Relative optimal transport was recently used for optimization in multiparameter persistent homology~\cite{Scoccola:2024}.

Topological properties of spaces of discrete measures equipped with relative transport distances were studied in \cite{BubenikHartsock, che2021metric}.
The authors of the present paper have studied universality properties of the space of persistence diagrams equipped with the relative Wasserstein distances \cite{bubenik2022universality, bubenik2022virtual}. 

The related study of unbalanced optimal transport has a well-developed theory~\cite{Hanin:1992,Hanin:1999,Guittet:2002,Benamou:2003,Savare:2024,
Piccoli:2014,Piccoli:2016,Piccoli:2023,Liero:2018,Liero:2016,Liero:2023,chizat2018unbalanced,chizat2018interpolating,Laschos:2019}.
A related but distinct problem is studied under the name partial optimal transport~\cite{caffarelli2010free,figalli2010optimal}. We note that the problem that we study here has also been referred to as partial optimal transport. We prefer the term relative optimal transport to distinguish it from the already established partial transport problem.

The interaction of cones, norms, and Riesz theory was also studied by Subramanian \cite{Subramanian:2012}.

In the late stages of preparing this paper we became aware of an independent work by Mauricio Che on optimal transport for metric pairs~\cite{Che:2024}.
Che restricts to measures on metric pairs $(X,d,A)$ whose support is contained in $X \setminus A$.
We note that it is easy to construct sequences of such measures that converge to our more general relative Radon measures.


\section{Background}
\label{sec:background}

  In this section we collect well known or elementary results that we will use in the sequel.
  We also use this section to fix notation.
All vector spaces will be real vector spaces.

\subsection{Ordered vector spaces} 
\label{sec:ordered-vector-spaces}

Let $V$ be a vector space.
A \emph{cone} in V is a subset $C\subset V$ such that $C+C\subset C$ and $a C\subset C$ for all $a \geq 0$. 
A cone $C$ is \emph{salient} if $C \cap -C = \{0\}$.
A cone $C$ is \emph{generating} if $C - C = V$.

A \emph{preordered vector space} is a vector space $V$ equipped with a preorder $\leq$ such that,
for all $x,y \in V$ with $x \leq y$, we have $x+z \leq y+z$ for all $z \in V$, and $a x \leq a y$ for all $a \geq 0$.
The set
$V^+ = \{x\in V \ | \ x\geq 0\}$
is a cone in $V$ called the \emph{positive cone}.
Conversely, given a cone $C$ in $V$,
$V$ is a preordered vector space under the preorder given by
$x \leq y$ if there exists $z \in C$ such that $x + z = y$,
and the positive cone of $(V,\leq)$ is $C$.
An \emph{ordered vector space} is a preordered vector space in which the preorder is a partial order.
A preordered vector space is an ordered vector space if and only if its positive cone is salient.
Let $V$ be an ordered vector space $V$, $A \subset V$ and $x \in V$.
Since addition by $x$ is an isomorphism of partially ordered sets,
$\sup_{a \in A}(x+a) = x + \sup A$ if either side is defined.

An \emph{operator} between preordered vector spaces $V$ and $W$ is a linear map $T:V\to W$.
Such an operator 
is \emph{positive} if for all $x\in V^+$,  $T(x)\in W^+$.
An operator is order preserving if and only if is positive.
The vector space of all operators from $V$ to $W$ will be denoted $\mathcal{L}(V,W)$. 
This becomes an preordered vector space with the 
preorder given by the cone of positive operators.
If both $V$ and $W$ are ordered vector spaces then so is $\mathcal{L}(V,W)$.
%
A \emph{morphism} between preordered vector spaces 
is
a positive operator, or, equivalently, an order preserving linear map.
A subset $A\subset V$
is \emph{order bounded} if there exists $x,y\in E$ such that $x\leq a \leq y$ for all $a\in A$. 
An operator $T:V\to W$
is said to be \emph{order bounded} if it maps ordered bounded subsets of $V$ to order bounded subsets of $W$. The operator $T$ is said to be \emph{regular} if it can be written as the difference of two positive operators.
Let $\mathcal{L}_{\textup{b}}(V,W)$ and $\mathcal{L}_{\textup{r}}(E,F)$ denote the subsets of $\mathcal{L}(V,W)$ consisting of the order bounded and regular operators, respectively.
Since positive operators are order-preserving, they are order bounded.
Therefore regular operators are likewise order bounded,
giving the inclusions $\mathcal{L}_{\textup{r}}(V,W) \subseteq \mathcal{L}_{\textup{b}}(V,W) \subseteq \mathcal{L}(V,W)$.

\subsection{Riesz spaces}
\label{sec:riesz}

A \emph{Riesz space} is a ordered vector space $E$ in which the poset structure forms a lattice. 
A Riesz space is also called a vector lattice.
That is, every pair $x,y\in E$
has a supremum $x\vee y$ and an infimum $x\wedge y$.
If $V$ is an ordered vector space for which $x \vee 0$ exists for each $x \in V$ then $V$ is a Riesz space,
since for $x,y \in V$, $x \vee y = y + (x-y) \vee 0$ and $x\wedge y = -(-x\vee -y)$.
A Riesz space is a distributive lattice.
A Riesz space is said to be \emph{Dedekind complete} (also called order complete) if every nonempty subset which is bounded above has a supremum. 
Equivalently, every nonempty subset which is bounded below has an infimum.
The real numbers form a Dedekind complete Riesz space under the usual ordering. 
A vector subspace $G$ of a Riesz space $E$ is a \emph{Riesz subspace} if for all $x,y \in G$, $x \join y \in G$.

For the rest of this section let $E$ and $F$ be Riesz spaces.
For $x\in E$, we define $x^+= x\vee 0$, $x^- = (-x)\vee 0$, and $|x| = x\vee(-x)$. 
Then $x^+, x^-, |x| \in E^+$,
$x = x^+ - x^-$, $|x| = x^+ + x^-$, and $x^+ \wedge x^- = 0$. 
The decomposition $x = x^+ - x^-$ is minimal in the sense that if $x = y-z$ for some $y,z \in E^+$ then $y\geq x^+$ and $z \geq x^-$. This decomposition is unique in the sense that if $x = y-z$ with $y\wedge z = 0$ then $y = x^+$ and $z = x^-$.
For $x,y\in E$, we write $[x,y] = \{z\in E \ | \ x\leq z\leq y\}$.
An element $e \in E^+$ is an \emph{order unit} of $E$ if for every $x \in E$, there is an $n \in \N$ such that $\abs{x} \leq ne$.

A map $T:E^+ \to F^+$ is \emph{additive} if for all $x,y \in E^+$, $T(x+y) = T(x) + T(y)$.
  The Riesz space $F$ is \emph{Archimedean} if for all $x \in E^+$, $\inf_{n \in \N}  \frac{1}{n} x  = 0$.
  It is a theorem of Kantorovich \cite[Theorem~1.10]{aliprantis2006positive_operators} that if
  $T:E^+ \to F^+$ is additive and
  $F$ is Archimedean,
  then $T$ has a unique extension to a positive operator $T:E \to F$ given by $T(x) = T(x^+) - T(x^-)$ for all $x \in E$.
From now on, we will assume that \emph{all of our Riesz spaces are Archimedean}.

A theorem of Riesz and Kantorovich \cite[Theorem~1.18]{aliprantis2006positive_operators} says that if $F$ is Dedekind complete then $\mathcal{L}_{\textup{b}}(E,F)$ is a Dedekind complete Riesz space.
Its lattice operations are given by
$(S\vee T)(x)  = \sup\{S(y) + T(z) \ | \ y + z = x, \, y,z\in E^+\}$ and
$(S\wedge T)(x) = \inf\{S(y) + T(z) \ | \ y + z = x, \, y,z\in E^+\}$
for all $x\in E^+$.
It follows that
if $F$ is Dedekind complete then 
$\mathcal{L}_{\textup{b}}(E,F) = \mathcal{L}_{\textup{r}}(E,F)$. 

A net $\{x_\alpha\}$ in 
$E$ is \emph{decreasing} if $\alpha \succeq \beta$ implies $x_\alpha \leq x_\beta$. The notation $x_\alpha \downarrow x$ means that $\{x_\alpha\}$ is decreasing and $x = \inf \{x_\alpha\}$. A net $\{x_\alpha\}$ in
$E$ is said to be \emph{order convergent} to $x\in E$, denoted $x_\alpha\stackrel{o}{\to} x$, if there exists a net $\{y_\alpha\}$ with the same index set satisfying $|x_\alpha- x|\leq y_\alpha$ and $y_\alpha \downarrow 0$.
A subset $A \subset E$ is \emph{solid} if for all $a \in A$ and all $x \in E$ with $\abs{x} \leq \abs{a}$, we have $x \in A$. 
An \emph{ideal} in $E$ is a solid linear subspace.
The identity $x \vee y = \frac{1}{2}(x + y + \abs{x-y})$ shows that an ideal is a Riesz subspace.
A subset $A \subset E$ is \emph{order closed} if whenever $\{x_{\alpha}\} \subset A$ and $x_{\alpha} \stackrel{o}{\to} x$ then $x \in A$.
A \emph{band} in $E$ is an order closed ideal.
An operator $T:E\to F$
is said to be  \emph{order continuous} if for any net $\{x_\alpha\}$ in $E$ with $x_\alpha\stackrel{o}{\to}  0$ we have $T(x_\alpha)\stackrel{o}{\to}  0$ in $F$. The operator $T$ is said to be 
\emph{sequentially order continuous} if for any sequence $(x_n)$ in $E$ with $x_n \stackrel{o}{\to} 0$ we have $T(x_n) \stackrel{o}{\to} 0$ in $F$.
If $T$ is positive then it is sequentially order continuous iff $x_n \downarrow 0$ implies $T x_n \downarrow 0$.
If $T$ is order bounded and $F$ is Dedekind complete, then the following are equivalent:
  $T$ is sequentially order continuous;
  for any sequence $(x_n)$ with $x_n\downarrow 0$, we have $T(x_n)\stackrel{o}{\to} 0$;
  $T^+$ and $T^-$ are both sequentially order continuous; and
  $\abs{T}$ is sequentially order continuous.
  Let
  $\mathcal{L}_{\textup{n}}(E,F)$ and
  $\mathcal{L}_{\textup{c}}(E,F)$ denote the subsets of
  $\mathcal{L}_{\textup{b}}(E,F)$ consisting of operators that are order continuous and sequentially order continuous, respectively. Thus
  $\mathcal{L}_{\textup{n}}(E,F) \subset
  \mathcal{L}_{\textup{c}}(E,F) \subset
  \mathcal{L}_{\textup{b}}(E,F)$.

  The \emph{order dual} of $E$ is given by $E^{\sim} = \mathcal{L}_{\textup{b}}(E,\R)$.
  Since $\R$ is a Dedekind complete Riesz space, it is the vector space generated by the positive linear functionals on $E$.
  The \emph{order continuous dual} of $E$ is given by
  $E_{\textup{n}}^{\sim} = \mathcal{L}_{\textup n}(E,\R)$.
  The \emph{sequentially order continuous dual} of $E$ is given by
  $E_{\textup{c}}^{\sim} = \mathcal{L}_{\textup c}(E,\R)$.
  We have
  $E_{\textup{n}}^{\sim} \subset
  E_{\textup{c}}^{\sim} \subset E^{\sim}$,
  and furthermore, both 
  $E_{\textup{n}}^{\sim}$ and
  $E_{\textup{c}}^{\sim}$ are bands in $E^\sim$.
  
  Say that $E^{\sim}$ \emph{separates the points} of $E$ if for each nonzero $x \in E$ there exists $f \in E^{\sim}$ with $f(x) \neq 0$.
  Since the order dual is a Riesz space, we have the \emph{second order dual} $E^{\sim \sim} = (E^{\sim})^{\sim}$.
  For each $x \in E$, we have the order bounded linear functional $\hat{x}: f \mapsto f(x)$.
  In fact, this linear functional is order continuous.
  If $E^{\sim}$ separates the points of $E$ then the mapping $x \to \hat{x}$ is one-to-one and embeds $E$ as a Riesz subspace of its second order dual.
  Furthermore, if $A$ is an ideal in $E^{\sim}$ that separates the points of $E$, then the mapping $x \to \hat{x}$ embeds $E$ as a Riesz subspace of $A_{\textup{n}}^{\sim}$.

\subsection{Monoids and the Grothendieck group completion}
\label{sec:monoids}

A \emph{commutative monoid} $M = (M,+,0)$ is a set $M$ together with an associative commutative binary operation $+:M\times M\to M$ for which there exists an element $0\in M$ satisfying $m+0 = m$ for all $m\in M$, 
called the \emph{neutral} element. $M$ is \emph{cancellative} if $a + c = b+ c$ implies $a = b$.
$M$ is \emph{zero-sum-free} if $a+b=0$ implies that $a=b=0$.
A \emph{monoid homomorphism} between commutative monoids $M = (M,+_M,0_M)$ and $N = (N,+_N,0_N)$ is a map $f:M\to N$ such that 
$f(a +_M b) = f(a) +_N f(b)$ for all $a,b\in M$ and
$f(0_M) = 0_N$.
A subset $P \subset M$ is a \emph{submonoid} if it contains $0$ and $+$ restricts to a binary operation on $P$.
%

A metric $\rho$ on a commutative monoid $M$ is \emph{translation invariant} if $\rho(a + c, b+c) = \rho(a,b)$ for all $a,b,c\in M$. Note that if $M$ is equipped with such a metric then $M$ is automatically cancellative.

An equivalence relation $\sim$ on a commutative monoid $M$ is called a \emph{congruence} if $a\sim b$ and $c\sim d$ implies $a+c \sim b + d$. If $\sim$ is a congruence then there is a well-defined commutative monoid structure on the set of equivalence classes $M/\!\!\sim$ given by $[a] + [b] = [a+b]$. 
%
Let $M$ be a commutative monoid and $P \subseteq M$ any submonoid. Define a relation $\sim$ on $M$ by 
$a\sim b$ iff there exist $x,y \in P$ such that $a+x = b+y$.
Then $\sim$ is a congruence and we denote the commutative monoid $M/\!\!\sim$ by $M/P$ and refer to it as the \emph{quotient of $M$ by $P$}. 

Given a commutative monoid $M = (M,+,0)$, the \emph{Grothendieck group} of $M$, denoted $K(M)$, is the abelian group defined as follows. Define an equivalence relation $\sim$ on $M\times M$ by $(a,b)\sim (a',b')$ if and only if there exists some $k\in M$ such that $a + b' + k = a' + b + k$. As a set, we define $K(M)= (M\times M)/\!\!\sim$. We denote the equivalence class of $(a,b)$ under $\sim$ by $a-b$. The binary operation on $K(M)$ is also denoted by $+$ and is defined by $(a-b) + (a'-b') = (a+a') - (b+b')$. This operation makes $K(M)$ into an abelian group with identity element $0 = 0-0$ and with the inverse of $a-b$ given by $b-a$. Note that if $M$ is a cancellative monoid then $a-b = a'-b'$ in $K(M)$ if and only if $a + b' = a' + b$ in $M$.
There is a canonical monoid homomorphism $i:M\to K(M)$ given by $m\mapsto m - 0$. If $M$ is cancellative then this map is injective and hence defines an embedding of $M$ into $K(M)$. The Grothendieck group is universal in the following sense. Given any abelian group $A$ and monoid homomorphism $f:M\to A$, there exists a unique group homomorphism $\tilde{f}:K(M)\to A$ such that $\tilde{f}\circ i = f$. Equivalently, the Grothendieck group construction gives rise to a functor $K:\cat{CMon}\to \cat{Ab}$ from the category of commutative monoids to the category of abelian groups, and this functor is left adjoint to the corresponding forgetful functor.

If $M$ is equipped with a translation invariant metric $\rho$, then $K(M)$ can be equipped with a canonical translation invariant metric $d$ given by $d(a-b,a'-b') = \rho(a+b',a'+b)$. In this case, $M$ is cancellative and $d$ restricts to $\rho$ on the image of $M$ in $K(M)$ under the canonical inclusion $i:M\hookrightarrow K(M)$. In categorical language, the functor $K$ restricts to a functor $K:\cat{CMon^{\textup{ti}}}\to\cat{Ab^{\textup{ti}}}$ from the full subcategories of $\cat{CMon}$ and $\cat{Ab}$ of commutative monoids and abelian groups, respectively, equipped with translation invariant metrics, and this functor is left adjoint to the corresponding forgetful functor \cite{bubenik2022virtual}.

\subsection{Convex cones} \label{sec:convex-cones}

Recall that $\R^+ = \{\alpha \in \R \ | \ \alpha \geq 0\}$.
A \emph{convex cone}~\cite{ConvexCones} is a 
commutative monoid $(C,+,0_C)$ together with a binary operation $\cdot:\R^+\times C\to C$ which satisfies, for all $\alpha,\beta\geq 0$ and $x,y,z\in C$,
\[\textup{(1)}\,\,  \alpha\cdot(x+y) = \alpha\cdot x + \alpha\cdot y, \quad \textup{(2)} \,\, (\alpha + \beta)\cdot x = \alpha\cdot x + \beta\cdot x, \quad \textup{(3)} \,\, (\alpha\beta)\cdot x = \alpha\cdot(\beta\cdot x),\]
\[\textup{(4)} \,\, 1\cdot x = x, \quad (5) \,\,
0\cdot x = 0_C.\]
A \emph{cone homomorphism} between convex cones $C, C'$ is a function $f:C\to C'$ such that $f(\alpha x + \beta y) = \alpha f(x) + \beta f(y)$ for all $x,y\in C$ and $\alpha,\beta \geq 0$. 

\begin{remark}
    We may define a vector space to be an $\R$-module and a convex cone to be an 
    $\R^+$-module, which are instances of the definition of a module over a commutative monoid internal to a  symmetric monoidal category. 
    In the former case, $\R$ is a commutative ring, i.e. a commutative monoid internal to $(\cat{Ab},\tensor,1)$, where $\tensor$ is Hassler Whitney's tensor product on abelian groups~\cite{Whitney:1938}. Similarly, $\R^+$ is a commutative semiring/rig, i.e. a commutative monoid internal to $(\cat{CMon},\tensor,1)$. In fact, these definitions are special cases of the definition of a module over a monad.
    Modules over semirings are also called semimodules~\cite{golan1999semirings}.
\end{remark}

%
There is a bijection between ordered vector spaces with generating positive cones and zero-sum-free cancellative convex cones.
Indeed, given an ordered vector space $(V,\leq)$, the positive cone $V^+$ is a zero-sum-free cancellative convex cone.
Given a zero-sum-free cancellative convex cone $(C,+)$, we have the vector space $K(C)$ and the partial order corresponding to the cone $C$.
Furthermore, this bijection respects sub-ordered vector spaces and sub-cones.
Note that ordered vector spaces with generating positive cones are the same as ordered vector spaces with a directed order.

An \emph{ordered convex cone} is an convex cone $C$ together with a partial order $\leq$ such that for $w,x,y,z \in C$ and $\alpha \in \R^+$, if $x \leq y$ and $w \leq z$ then $x+w \leq y+z$ and $\alpha x \leq \alpha y$.
A \emph{lattice cone} \cite{ConvexCones} is an ordered convex cone $(C,\leq)$ such that 
for all $x,y \in C$ there is a supremum $x \vee y$ and
for all $x,y,z \in C$, $(x \vee y) + z = (x+z) \vee (y+z)$.
A lattice cone need not be a lattice.
A convex cone $C$ has a \emph{natural preorder} $\leq$ given by $x \leq y$ iff there exists $z \in C$ such that $x+z = y$.
If $C$ is cancellative and zero-sum-free then $\leq$ is a partial order.

A \emph{norm} on a convex cone $C$ is a metric $\rho$ on $C$ satisfying, for all $x,y\in C$ and $\alpha \geq 0$,
$\rho(\alpha x, \alpha y) = \alpha\rho(x,y)$ ($\R^+$-homogeneity), and
$\rho(x+z,y+z) = \rho(x,y)$ (translation invariance).
Such a norm is \emph{subadditive}: by the triangle inequality and translation invariance,
$\rho(x+y,x'+y')\leq \rho(x + y,x' + y) + \rho(x' + y, x'+y') = \rho(x,x') + \rho(y,y')$.
 A pair $(C,\rho)$, where $C$ is a convex cone and $\rho$ is a norm on $C$ is called a \emph{normed convex cone}. 
Such a convex cone is a cancellative.
 Indeed, if $x+z = y+z$ then, by translation invariance of $\rho$, we have $0 = \rho(x+z, y+z) = \rho(x,y)$ and hence $x = y$.
  A cone homomorphism between normed convex cones $\Phi:(C,\rho)\to (C',\rho')$ is said to be \emph{bounded} if $\Phi$ is Lipschitz. 

\begin{remark} 
  A norm on a convex cone resembles a vector space norm in the following sense. Given a normed convex cone $(C,\rho)$, define $\norm{\cdot}_\rho:C\to \R$ by $\norm{x}_\rho = \rho(x,0_C)$ for all $x\in C$. Then $\norm{\cdot}_\rho$ is positive definite, $\R^+$-homogeneous, 
		and satisfies the triangle inequality, analogous to a vector space norm. However, it is not possible, in general, to recover $\rho$ from $\norm{\cdot}_\rho$ as is the case for vector space norms. 
\end{remark}


Let $V$ be a vector space.
Given a cone $C$ in $V$, the vector space operations define a cancellative convex cone structure on $C$.
If $V$ is equipped with a norm $\norm{\cdot}$, then $C$ becomes a normed convex cone when equipped with the restriction of the metric induced by $\norm{\cdot}$. 

Conversely, 
given a convex cone $C$, let $K(C)$ denote its Grothendieck group. Then $K(C)$ can be equipped with a vector space structure by defining scalar multiplication by $\alpha  (x-y) = \alpha x - \alpha  y$ if $\alpha \geq 0$ and $\alpha (x-y) = |\alpha| y - |\alpha| x$ otherwise. 
If $C$ is cancellative then
the canonical inclusion $C\to K(C)$ is an injective cone homomorphism and hence the convex cone operations on (the image of) $C$ are obtained by restriction. Moreover, if $C$ is 
a normed cone,
then $\rho$ extends canonically to an absolutely $\R$-homogeneous, translation invariant metric $d$ on $K(C)$ \cite{bubenik2022virtual}. 
These are exactly the conditions needed for a metric to be induced by a norm $\norm{-}$.
These are related as follows, $\norm{x-y} = d(x-y,0) = \rho(x,y)$.

Combining the above results, we have a bijection between normed ordered vector spaces with generating positive cones and zero-sum-free normed convex cones, given by sending a vector space to its positive cone, and by sending a zero-sum-free normed convex cone $(C,\rho)$ to the vector space $K(C)$ with norm given by $\norm{x-y} = \rho(x,y)$.
Furthermore, this bijection respects sub-ordered vector spaces and sub-cones.

 \subsection{Metric pairs and Lipschitz functions}
\label{sec:lipschitz}

Let $(X,d)$ be a metric space with a closed subset $A$.
We call this a \emph{metric pair} and denote it by $(X,d,A)$, or more by simply by $(X,A)$.
We will consider $\R$ to be a metric space with the usual metric given by $d(x,y) = \abs{x-y}$ and a metric pair with the subset $\{0\}$.
We also have the sub-metric pair $\R^+$.

Given a metric space $(X,d)$, a subset $A\subset X$, and 
$\eps \geq 0$, we define the \emph{$\eps$-offset} of $A$ by $A^{\eps} = \{x\in X \ | \ d(x,A) \leq \eps\}$. Let $A^{\infty}$ denote $X$.

A metric space $(X,d)$ is \emph{boundedly compact} if it has the Heine-Borel property: every closed and bounded subset is compact. Equivalently, a metric space is boundedly compact if every closed ball is compact.
Such metric spaces are also called proper.
A $\sigma$-compact space is Lindel\"of.
For a metric space, the properties Lindel\"of, separable, and second-countable are equivalent.

For $L\geq 0$, a function $f:(X,d)\to (Y,e)$ between metric spaces is said to be \emph{$L$-Lipschitz} 
if $e(f(x),f(x')) \leq L d(x,x')$ for all $x,x'\in X$.
The function $f$ is said to be \emph{Lipschitz} if it is $L$-Lipschitz for some $L \geq 0$.
Call the smallest such constant the \emph{Lipschitz number} of $f$ and  denote it by $L(f)$.
Given functions
$f:X \to \R$ and $y: Y \to \R$, define $f \oplus g: X \times Y \to \R$ by
$f \oplus g = f \circ p_1 + g \circ p_2$, where $p_1:X \times Y \to X$ and $p_2:X \times Y \to Y$ denote the canonical projections. 
That is, for $(x,y) \in X \times Y$,
\begin{equation} \label{eq:oplus}
  (f \oplus g)(x,y) = f(x) + g(y).
\end{equation}
A function $f:X \to \R$ is $L$-Lipschitz if and only if $f \oplus (-f) \leq L d$.

A morphism of metric pairs $f:(X,d,A) \to (Y,e,B)$ is a Lipschitz function $f:(X,d) \to (Y,e)$ such that $f(A) \subset B$.
Given metric pairs
$(X,d,A)$ and $(Y,e,B)$,
define the
product of these metric spaces to be the metric pair $(X \times Y, d + e, A \times B)$, where $(d+e)((x,y),(x',y')) = d(x,x') + e(y,y')$.
This is the categorical product in 
the category of metric pairs and morphisms of metric pairs,
and the canonical projection morphisms are $1$-Lipschitz. 
  
Let $(X,d,A)$ be a metric pair.
Let $\LipX$ denote the set of morphisms of metric pairs from $(X,d,A)$ to $\R$. 
If $A = \emptyset$ then $\LipX$ is the set of Lipschitz functions from $X$ to $\R$.
The Lipschitz number is only a semi-norm on Lipschitz functions
on $X$, but
if $A \neq \emptyset$ then
it is a norm on $\LipX$. 

If $A \neq \emptyset$ then
the vector space $\LipX$ together with the Lipschitz number $(\LipX, \lipnorm{\cdot})$ is a Banach space~\cite[Proposition 2.3(b)]{weaver}. 
Denote the collection of all $f\in \LipX$ with compact support by $\LipcX$.
The vector space $\LipcX$ with the Lipschitz number is a normed vector space but need not be a Banach space.
For $L \geq 0$, we define $\LipLX = \{f\in \LipX \ | \ \lipnorm{f}\leq L\}$ and $\LipcLX = \LipcX\cap \LipLX$. 
Then $\LipOneX$ and $\LipcOneX$ are the closed unit balls of $\LipX$ and $\LipcX$, respectively.
Similarly, we denote the set of morphisms of metric pairs from $(X,d,A)$ to $\R^+$ by $\LipXpos$.
It has the subsets $\LipLXpos$, $\LipcXpos$, and $\LipcLXpos$.

If $f,g:X\to \R$ are bounded Lipschitz functions then their product $fg$ is Lipschitz as well. If $h,k:X\to \R$ are Lipschitz and one of $h$ or $k$ is compactly supported, then $hk$ is also compactly supported and Lipschitz.
For a function $f:X\to \R$, we define $f^+ = \max(0,f)$ and $f^- = -\min(0,f) = (-f)^+$. Then $f = f^+ - f^-$, and if $f$ is Lipschitz then so are $f^+,f^-$ with $\lipnorm{f^+},\lipnorm{f^-}\leq \lipnorm{f}$.
Given a 
subset $B \subset (X,d)$, 
we define $d_B:X\to \R$ by $d_B(x) = d(x,B)$. 
Then $d_B$ is $1$-Lipschitz and $d_B(x) = 0$ for all $x\in B$. 
Moreover, $d_B \geq d_C$ whenever $B\subset C$ and $d_{\overline{B}} = d_B$ for all $B\subset X$.
For $x \in X$, denote $d_{\{x\}}$ by $d_x$.
If $A \neq \emptyset$ then
$d_A \in \LipOneXpos$. If $A = \emptyset$, then we adopt the convention that $d_A(x) = \infty$ for all $x\in X$.

We equip $\LipX$ with the \emph{pointwise partial order}, i.e., for $f,g\in \LipX$, we have $f\leq g$ iff $f(x)\leq g(x)$ for all $x\in X$. This partial order makes $\LipX$ into a Riesz space, with $f\vee g = \max(f,g)$ and $f\wedge g = \min(f,g)$.
Note that $\lipnorm{f\vee g}\leq \max(\lipnorm{f},\lipnorm{g})$ and $\lipnorm{f\wedge g}\leq \max(\lipnorm{f},\lipnorm{g})$.
The positive cone of $\LipX$ is $\LipXpos$.
Recall that the Riesz space $\LipX$ together with the Lipschitz number is also a Banach space. However, this norm is not a \emph{lattice norm}, since $\abs{f} \leq \abs{g}$ does not imply that $\lipnorm{f} \leq \lipnorm{g}$.
So $\LipX$ is not a Banach lattice or a normed Riesz space.
If $f,g$ are compactly supported then so is $f\vee g$, and hence $\LipcX$ is a Riesz subspace of $\LipX$.
If $f$ is compactly supported and $\abs{g} \leq \abs{f}$ then $g$ is compactly supported.
  Therefore $\LipcX$ is an ideal in $\LipX$.
  However, $\LipcX$ need not be a band in $\LipX$.
Assume $A \neq \emptyset$.
For $f \in \LipX$, 
since $f(A) = 0$, $\abs{f} \leq \lipnorm{f} d_A$.
Therefore the function $d_A$ is an order unit for $\LipX$.

\subsection{Measure theory}
\label{sec:measure-theory}

A \emph{measure} is a countably additive set function on a $\sigma$-algebra with values in $[0,\infty]$.
A measure is \emph{finite} if it has values in $[0,\infty)$.
Let $X$ be a Hausdorff topological space and let $\mu$ be a Borel measure on $X$.
The measure $\mu$ is \emph{tight} if it is inner regular with respect to compact sets, i.e. for all Borel sets $E$,
  $\mu(E) = \sup\{\mu(K) \ | \ K \text{ compact and } K\subset E\}$.
Equivalently, for all Borel sets $E$ and $\eps > 0$
there is a compact set $K_{\eps} \subset E$ such that $\mu(E \setminus K_{\eps}) < \eps$.
The measure $\mu$ is 
\emph{locally finite} if each $x \in X$
has some neighborhood $A$ with $\mu(A)<\infty$. 
The 
measure $\mu$ 
is a \emph{Radon measure} if it is tight and locally finite.
The measure $\mu$ is \emph{$\tau$-additive} if whenever $\{U_{\alpha}\}$ is an upwards-directed family of open sets then $\mu(\bigcup_{\alpha} U_{\alpha}) = \sup_{\alpha} \mu(U_{\alpha})$.
If $\mu$ is tight then $\mu$ is $\tau$-additive~\cite[414E]{fremlin4}.

If $E$ is a Borel subset of $X$ then the Borel subsets of $E$ with respect to the subspace topology are exactly the Borel subsets of $X$ that are contained in $E$. Therefore the restriction of $\mu$ to these Borel sets, denoted $\mu_E$ is a Borel measure on $E$.
This Borel measure on $E$ has a canonical extension to Borel measure on $X$, which we will also denote by $\mu_E$, given by $\mu_E(B) = \mu_E(B \cap E) = \mu(B \cap E)$, for any Borel subset $B$ of $X$.
%
A Borel
measure $\mu$ on 
$X$ defines a positive linear functional on
any vector space of $\mu$-integrable functions
given by $f\mapsto \int_Xfd\mu$. 
We also denote this linear functional by $\mu$, so that $\mu(f) = \int_X fd\mu$.

A \emph{signed measure} is a countably additive set function on a $\sigma$-algebra with values in $\R$.
A signed measure $\mu$ has a Jordan decomposition $\mu = \mu^+ - \mu^-$, where $\mu^+$ and $\mu^-$ are finite measures. 
The \emph{variation} of $\mu$ is given by $\abs{\mu} = \mu^+ + \mu^-$,
which is finite.
A signed Borel measure $\mu$ is a \emph{signed Radon measure} if $\abs{\mu}$ is tight.

\subsection{Extensions of positive linear functionals}
\label{sec:hahn-banach}

We will use the following extension theorems.
The first is a classical result of Kantorovich and the second is a consequence of the Hahn-Banach theorem.

\begin{theorem}[\cite{Kantorovich:1937}] \label{prop:positive_hahn_banach_weak}
Let $V$ be an ordered vector space.
  Let $W\subset V$ be a subspace with the property that 
for all $v\in V$, there exists $w\in W$ with $v \leq w$.
Then any positive linear functional $T:W \to \R$ has an extension to a positive linear functional $T':V \to \R$.
\end{theorem}


\begin{theorem}[{\cite[Theorem 1.27]{aliprantis2006positive_operators}}] \label{thm:positive_hahn_banach_strong}
    Let $E$ be a Riesz space with Riesz subspace $G$ and let $T:G \to \R$ be a positive linear functional.
    Then $T$ has an extension to a positive linear functional $T':E \to \R$ if and only if 
    there exists a monotonic sublinear functional $\rho:E \to \R$ such that for all $x \in G$, $T(x) \leq \rho(x)$.
\end{theorem}

\section{Riesz cones} 
\label{sec:riesz-cone}

In this section we develop analogs of Riesz spaces 
whose underlying structure is a convex cone rather than a vector space.
Taking the Grothendieck group we obtain Riesz spaces. 

\begin{definition}
    A \emph{Riesz cone} is a 
    cancellative convex cone that it is a lattice cone with respect to the natural partial order.
\end{definition}

Recall that a lattice cone has pairwise suprema but need not have pairwise infima. However, we will show that Riesz cones do indeed have pairwise infima (\cref{prop:riesz-cone-meet}).

\begin{lemma}
    Let $C$ be a Riesz cone and let $x,y \in C$ such that $x \leq y$.
    Then there exists a unique $z \in C$ such that $x+z = y$.
\end{lemma}

\begin{proof}
    By the definition of the natural partial order, there is a $z \in C$ such that $x+z = y$.
    Assume there exists $w \in C$ such that $x+w=y$.
    Then $x+w = x+z$, and since $C$ is cancellative, $w=z$.
\end{proof}

\begin{lemma}
    Let $C$ be a Riesz cone and let $x,y,z \in C$ such that $x+z \leq y+z$. Then $x \leq y$.
\end{lemma}

\begin{proof}
    By the definition of the natural partial order, 
    there is a $w \in C$ such that $x+z+w=y+z$.
    Since $C$ is cancellative,
    $x+w=y$.
    Therefore $x \leq y$.
\end{proof}

\begin{proposition} \label{prop:riesz-cone-meet}
    Let $C$ be a Riesz cone.
    For each $x,y \in C$ there is exists an infimum $x \wedge y$ such that for all $x,y,z \in C$,
    $(x \wedge y) + z = (x+z) \wedge (y+z)$.
    In addition, for all $x,y \in C$, $x \vee y + x \wedge y = x + y$.
\end{proposition}

\begin{proof}
    Let $x,y \in C$.
    Since $x,y \leq x+y$, 
    $x\vee y \leq x+y$.
    Therefore, there exists $z \in C$ such that $x \vee y + z = x+y$.
    That is, $(x+z) \vee (y+z) = x + y$.
    Thus $x+z \leq x+y$ and hence $z \leq y$.
    Similarly $z \leq x$ and
    hence $z$ is a lower bound for $\{x,y\}$.
    Let $w$ be a lower bound for $\{x,y\}$.
    Then $w+y \leq x+y$ and $w+x \leq x+y$.
    Thus $w + x \vee y = (w+x) \vee (w+y) \leq x+y = x \vee y + z$.
    Hence $w \leq z$.    
    Therefore $z = x \wedge y$ and $x \vee y + x \wedge y = x+y$.

    Let $x,y,z \in C$.
    We have that 
    $x \vee y +z + (x+z) \wedge (y+z) = (x+z) \vee (y+z) +(x+z) \wedge (y+z) = x + y + 2z = x \vee y + x \wedge y + 2z$.
    Thus $(x+z) \wedge (y+z) = x \wedge y + z$.
\end{proof}

Note that $\vee$ and $\wedge$ are monotone in either coordinate.

\begin{proposition} \label{prop:riesz-cone-minus}
    Let $C$ be a Riesz cone.
    For each $x,y \in C$, there is a unique element in $C$ denoted $x \setminus y$ such that $y + x \setminus y = x \join y$. In addition, for all $x,y \in C$, $ x \meet y + x \setminus y = x$.
    Furthermore, $\setminus$ is monotone in the first coordinate.
\end{proposition}

\begin{proof}
    Let $x,y \in C$.
    Since $y \leq x \vee y$, there is a unique $z \in C$ such that $y+z = x \vee y$.
    Denote $z$ by $x \setminus y$.
    Since $x \wedge y \leq x$, there is a unique $w \in C$ such that $x \wedge y + w = x$.
    Then $x + y + x \setminus y 
    = x \wedge y + w  + x \vee y
    = x+y+w$.
    Therefore $w = x \setminus y$.

    Let $x,x',y \in C$ with $x \leq x'$.
    Then $x \vee y \leq x' \vee y$.
    That is, $y + x \setminus y \leq y + x' \setminus y$.
    Therefore $x \setminus y \leq x' \setminus y$.
\end{proof}

\begin{proposition}
    Let $C$ be a Riesz cone.
    Then $C$ is a distributive lattice.
\end{proposition}

\begin{proof}
    Let $x,y,z \in C$.
    We show that $x \meet (y \join z) = (x \meet y) \join (x \meet z)$.
    Since $y,z \leq y \join z$,
    $x \meet y \leq x \meet (y \join z)$ and $x \meet z \leq x \meet (y \join z)$.
    Thus $x \meet (y \join z)$ is an upper bound for $\{x \meet y, x \meet z\}$.
    Let $w$ be an upper bound for $\{x \meet y, x \meet z\}$.
    By \cref{prop:riesz-cone-minus}, $y = x \meet y + y \setminus x \leq w + (y \join z) \setminus x$ and similarly for $z$.
    Thus $y \join z \leq w + (y \join z) \setminus x$.
    Therefore $y \join z + x \meet (y \join z) \leq w + y \join z$ and 
    hence $x \meet (y \join z) \leq w$.
    Therefore $x \meet (y \join z)$ is the desired supremum.
    It follows that $x \join (y \meet z) = (x \join y) \meet (x \join z)$.
\end{proof}

\begin{proposition}
\label{lem:Riesz-space-cone}
    There is a bijection between Riesz spaces 
    and zero-sum-free Riesz cones given by sending a Riesz space to its positive cone, and by sending a zero-sum-free Riesz cone $C$ to the ordered vector space $K(C)$ with partial order given by the cone $C$ and defining for $x-y \in K(C)$, $(x-y) \vee 0 = x \vee y - y$.
\end{proposition} 

\begin{proof}
    Let $E$ be a Riesz space.
    The positive cone of an ordered vector space is a zero-sum-free cancellative convex cone.
    Since $E$ is a Riesz space, $E^+$ is a Riesz cone.
    Let $C$ be a  zero-sum-free Riesz cone and let $x-y \in K(C)$.
    Since $K(C)$ is an ordered vector space, 
    $\sup(x-y,0) = \sup(x,y)-y$ if either side exists.
\end{proof}

\begin{definition}
    Let $C$ be a zero-sum-free Riesz cone. 
    A subset $A \subset C$ is \emph{solid} if for all $y \in A$ and $x \in C$ with $x \leq y$, $x \in A$.
    An \emph{ideal} in $C$ is a solid Riesz subcone.
\end{definition}

\begin{proposition} \label{prop:bijection-ideal}
    The bijection between Riesz spaces and zero-sum-free Riesz cones gives a bijection between Riesz spaces  and their ideals and zero-sum-free Riesz cones and their ideals.
\end{proposition} 

\begin{proof}
    Let $E$ be a Riesz space with ideal $A$. 
    By \cref{lem:Riesz-space-cone}, $E^+$ is a zero-sum-free Riesz cone with Riesz subcone $A^+$.
    Let $y \in A^+$ and let $x \in E^+$ such that $x \leq y$.
    That is, $\abs{x} \leq \abs{y}$. Therefore $x \in A$ and hence $x \in A^+$.

    Let $C$ be a zero-sum-free Riesz cone with ideal $A$.
    By \cref{lem:Riesz-space-cone}, $K(C)$ and $K(A)$ are Riesz spaces.
    Recall that $K(C) = (C \times C)/\sim_C$ and $K(A) = (A \times A)/\sim_A$.
    Since $C$ is cancellative, $\sim_C$ restricts to $\sim_A$ on $A \times A$.
    Therefore $K(A)$ is a Riesz subspace of $K(C)$.
    It remains to show that $K(A)$ is solid.
    Let $x-y \in K(A)$ and $w-z \in K(C)$ with $\abs{w-z} \leq \abs{x-y}$.
    That is, $w+z \leq x+y$.
    Thus $w,z \leq x+y \in A$. 
    Hence $w,z \in A$ and therefore $w-z \in K(A)$.
\end{proof}

\section{Measures for metric pairs}
\label{sec:measures}

In this section we define Radon measures for metric pairs and determine some of their structure.

\subsection{Borel measures on metric pairs}
\label{sec:borel-pointed}
 
Let $(X,d,A)$ be a metric pair.
That is, $(X,d)$ is a metric space and $A \subset X$ is a closed subspace.
Let $\pBorelXonly$ denote the set of Borel measures on $X$.
Addition and the zero measure give $\pBorelXonly$ the structure of a commutative monoid.
It is zero-sum-free, but not cancellative. 
Let $\mu \in \pBorelXonly.$
The Borel sets of $A$ are the Borel sets of $X$ that are contained in $A$ and $\mu$ restricts to $\mu_A \in \pBorelAonly$.
Furthermore, $\pBorelAonly$ is a submonoid of $\pBorelXonly$.
Similarly, $\mu$ restricts to $\mu_{X \setminus A}$ and $\mu = \mu_A + \mu_{X \setminus A}$.

\begin{definition}
    Let $\pBorel$ be the quotient monoid $\pBorelXonly / \pBorelAonly$.
    Call elements of $\pBorel$ \emph{relative Borel measures} on $(X,A)$, or, more simply, \emph{Borel measures} on $(X,A)$.
\end{definition}

\begin{lemma} \label{lem:borel-uniqueness}
    Let $\mu,\nu \in \pBorelXonly$.
    Then  $[\mu] = [\nu] \in \pBorel$ if and only if $\mu_{X \setminus A} = \nu_{X \setminus A}$.
\end{lemma}

\begin{proof}
    Assume $[\mu] = [\nu] \in \pBorel$.
    Then there exists $\sigma, \tau \in \pBorelAonly$ such that $\mu + \sigma = \nu + \tau$.
    Let $E$ be a Borel set in $X \setminus A$.
    Then $\mu(E) = (\mu + \sigma)(E) = (\nu + \tau)(E) =  \nu(E)$. 
    Therefore 
    $\mu_{X \setminus A} = \nu_{X \setminus A}$.

    Assume $\mu_{X \setminus A} = \nu_{X \setminus A}$.
    Then $\mu + \nu_{A}
    = \mu_{X \setminus A} + \mu_A + \nu_A
    = \nu_{X\setminus A} + \mu_A + \nu_A 
    = \nu + \mu_A$.
    Since $\mu_A, \nu_A \in \pBorelAonly$, $[\mu] = [\nu]$.
\end{proof}

Hence, there is a bijection between 
Borel measures on $(X,A)$ and
Borel measures on $X \setminus A$ that sends $[\mu]$ to $\mu_{X \setminus A}$.
We will use this bijection implicitly.
If $f \in \LipX$ then $\int_X f d\mu = \int_{X \setminus A} f d\mu_{X \setminus A}$.
Thus, for $[\mu] \in \pBorel$ and $f \in \LipX$, $\int_X f d\mu \in [0,\infty]$ is well defined.
If a Borel measure $\mu$ on $X$ is tight then so is $\mu_{X \setminus A}$.
We say that $[\mu] \in \pBorel$ is \emph{tight} if $\mu|_{X \setminus A}$ is tight, which is well defined by \cref{lem:borel-uniqueness}.
Similarly, for $[\mu] \in \pBorel$, the property of being locally finite at 
$x \in X \setminus A$,
  is well defined. 
For $[\mu] \in \pBorel$, let the \emph{support} of $[\mu]$, be defined by
\begin{equation*}
    \supp([\mu]) = \{ x \in X \ | \ \text{for each open neighborhood $U$ of $x$, } \mu(U \cap (X \setminus A)) \neq 0\}.
\end{equation*}
Note that $\supp([\mu]) \supset \supp(\mu|_{X \setminus A})$ but these need not be equal, since the former may contain points of $A$. If we view $\supp(\mu|_{X\setminus A})$ as a subset of $X$, then $\supp([\mu]) = \overline{\supp(\mu|_{X\setminus A})}$.

For simplicity, from now on we will use $\mu$ instead of $[\mu]$ to denote elements of $\pBorel$.
We will need the following uniqueness result.

\begin{lemma} \label{prop:uniqueness}
  Let $\mu,\nu \in \pBorel$ such that
  $\mu$ and $\nu$ are tight and
$\mu(K) = \nu(K)$
  for all compact subsets 
  $K \subset X \setminus A$.
  Then $\mu = \nu$.
\end{lemma}

\begin{proof}
A subset $U \subset X \setminus A$ is compact in $X \setminus A$ if and only if it is compact as a subset of $X$.
    Since $\mu$ and $\nu$ agree on compact subsets of $X \setminus A$, it follows from the definition of tightness that 
    $\mu_{X \setminus A} = \nu_{X \setminus A}$.
    Therefore $\mu = \nu$.
\end{proof}

\subsection{Finiteness conditions on Borel measures}
\label{sec:finiteness-conditions}

Let $(X,d,A)$ be a metric pair.
Assume that $A \neq \emptyset$.
For $\eps \geq 0$, let $A^{\eps} = \{x \in X \ | \ d_A(x) \leq \eps\}$ and $A^{\infty} = X$.
Let $0 \leq \eps \leq \delta \leq \infty$. 
Let $A_\eps^\delta = A^\delta \setminus A^\eps$ and let $A_\eps = A_\eps^\infty$.
In particular, $A^0 = A$ and $A_0^\infty = X \setminus A$.
Let $\mu$ be a Borel measure on $X$.
Define a Borel measure $\mu_\eps^\delta$ on $X$ by setting
\begin{equation*}
    \mu_\eps^\delta(E) = \mu( E \cap A_\eps^\delta),
\end{equation*}
for each Borel set $E$.
Also let 
$\mu^{\eps}$ denote $\mu_{0}^{\eps}$ and let
$\mu_{\eps}$ denote $\mu_{\eps}^{\infty}$.
The measure $\mu_\eps^\delta$ is well defined for $[\mu] \in \pBorel$.
Also there is a bijection between $[\mu] \in \pBorel$ and measures of the form $\mu_0^{\infty} \in \pBorelXonly$.
Recall that we will often abuse notation and refer to both $[\mu]$ and $\mu_0^{\infty}$ by $\mu$.
Note that for $a<b<c$, $A_a^b \cup A_b^c = A_a^c$ and $\mu_a^b + \mu_b^c = \mu_a^c$.

Let $\mu \in \pBorel$ and let $0 \leq p < \infty$.
Say the $\mu$ is \emph{upper $p$-finite} if for all $\eps > 0$, $\mu_{\eps}(d_A^p) < \infty$.
Say that $\mu$ is \emph{upper finite} if $\mu$ is upper $0$-finite, i.e. for all $\eps>0$, $\mu_{\eps}(X) < \infty$.
Say that $\mu$ is \emph{upper $\infty$-finite} if $\mu$ is upper finite and if there exists $\delta>0$ such that $\mu_{\delta} = 0$. 
That is, $\mu$ is upper finite and the essential supremum of $d_A$ with respect to $\mu$ is finite.

\begin{lemma} \label{lem:upper-p-finite}
    If $\mu \in \pBorel$ is upper $p$-finite, then as $\eps \to \infty$, $\mu_\eps(d_A^p) \downarrow 0$.
\end{lemma}

\begin{proof}

    Since $\mu$ is upper $p$-finite, $\mu_1^{\infty}(d_A^p) < \infty$. 
    By the monotone convergence theorem, as $\eps \to \infty$, $\mu_1^\eps(d_A^p)$ increases to $\mu_1^\infty(d_A^p)$.
    Since $\mu_1^\infty = \mu_1^\eps + \mu_\eps$, the result follows.
\end{proof}

\begin{lemma} \label{lem:upper-finite-q-p}
    Let $\mu \in \pBorel$ and let $0 \leq p \leq q \leq \infty$.
    If $\mu$ is upper $q$-finite then $\mu$ is upper $p$-finite.
\end{lemma}

\begin{proof}
    We start with the case that $\mu$ is upper $\infty$-finite and $0 \leq p < \infty$.
    Let $\eps > 0$.
    Since $\mu$ is upper $\infty$-finite, there is a $\delta \geq \eps$ such that $\mu_{\delta} = 0$.
    For all $x \in A_\eps^\delta$, $d_A(x) \leq \delta$.
    Then $\mu_{\eps}(d_A^p) = \mu_{\eps}^{\delta}(d_A^p) \leq \delta^p \mu_{\eps}^{\delta}(X) \leq \delta^p \mu_{\eps}(X) < \infty$ since $\mu$ is upper finite.

    Assume that $\mu$ is upper $q$-finite for $q < \infty$ and let $0 \leq p \leq q$.
    Let $\eps \geq 1$. 
    For all $x \in A_1^\infty$, $d_A^p(x) \leq d_A^q(x)$.
    Therefore $\mu_{\eps}(d_A^p) \leq \mu_{\eps}(d_A^q) < \infty$.
%
    Let $0 < \eps < 1$.
    We have that $\mu_{\eps} = \mu_{\eps}^1 + \mu_1$.
    Let $0 \leq r < \infty$.
    For $x \in A_\eps^1$, $\eps < d_A(x) \leq 1$.
    Thus $\eps^r \mu(A_\eps^1) \leq \mu_{\eps}^1(d_A^r) \leq \mu(A_\eps^1)$. 
    Hence $\mu_\eps^1(d_A^r) \leq \mu(A_1^\eps) \leq \frac{1}{\eps^r} \mu_\eps^1(d_A^r)$.
    Therefore
    $\mu_{\eps}^1(d_A^p) \leq \mu(A_\eps^1) \leq \frac{1}{\eps^q} \mu_{\eps}^1(d_A^q) < \infty$.
    Hence $\mu_{\eps}(d_A^p)  = \mu_{\eps}^1(d_A^p) + \mu_1(d_A^p) < \infty$.
\end{proof}

Let $\mu \in \pBorel$ and let $0 \leq p < \infty$.
Say that $\mu$ is \emph{lower $p$-finite} if for all $\eps > 0$, $\mu^{\eps}(d_A^p) < \infty$.
Say that $\mu$ is \emph{lower finite} if $\mu$ is lower $0$-finite, i.e. for all $\eps > 0$, $\mu^{\eps}(X) < \infty$.
Say that $\mu$ is \emph{lower $\infty$-finite} if for all $0 < \delta < \eps$, $\mu_{\delta}^{\eps}(X) < \infty$.

\begin{lemma} \label{lem:lower-p-finite}
    Let $0 \leq p < \infty$.
    If $\mu \in \pBorel$ is lower $p$-finite, then as $\eps \downarrow 0$, $\mu^\eps(d_A^p) \downarrow 0$.
\end{lemma}

\begin{proof}
    For $0 < \eps < 1$, $\mu^1(d_A^p) = \mu^\eps(d_A^p) + \mu_\eps^1(d_A^p)$.
    As $\eps \downarrow 0$, $\mu_\eps^1(d_A^p) \uparrow \mu^1(d_A^p) < \infty$.
    Therefore, as $\eps \downarrow 0$, $\mu^\eps(d_A^p) \downarrow 0$.
\end{proof}

\begin{lemma}
    Let $\mu \in \pBorel$ and $0 \leq p \leq q \leq \infty$. 
    If $\mu$ is lower $p$-finite then $\mu$ is lower $q$-finite.
\end{lemma}

\begin{proof}
    Assume that $\mu$ is lower $p$-finite and $p < \infty$.
    First we show that $\mu$ is lower $\infty$-finite.
    Suppose $0 < \delta < \eps$.
    Let $x \in A_\delta^\eps$.
    Then $d_A(x) > \delta$, which implies that $1 \leq \frac{1}{\delta^p} d_A^p(x)$.
    Therefore $\mu_{\delta}^{\eps}(X) \leq \frac{1}{\delta^p} \mu_\delta^\eps(d_A^p) \leq \frac{1}{\delta^p} \mu^\eps(d_A^p) < \infty$.

    Now assume that $q<\infty$.
    Suppose $0 < \eps \leq 1$.
    For $x \in A^{\eps}$, $d_A^q(x) \leq d_A^p(x)$.
    Thus $\mu^{\eps}(d_A^q) \leq \mu^{\eps}(d_A^p) < \infty$.
    Suppose $\eps > 1$.
    Then $\mu^{\eps} = \mu^1 + \mu_1^{\eps}$.
    For $x \in A_1^\eps$, $1 < d_A(x) \leq \eps$.
    Therefore $\mu(A_1^\eps) < \mu_1^{\eps}(d_A^p) \leq \eps^{p} \mu(A_1^\eps)$.
    Hence $\mu_1^\eps(d_A^q) \leq \eps^q \mu(A_1^\eps) < \eps^q \mu_1^\eps(d_A^p) \leq \eps^q \mu^\eps(d_A^p) < \infty$.
\end{proof}

The strongest combination of these conditions is that $\mu$ is lower $0$-finite and upper $\infty$-finite, which is equivalent to saying that $\mu$ is finite and has bounded support.
The weakest combination of these conditions is that $\mu$ is lower $\infty$-finite and upper $0$-finite, which is equivalent to saying that $\mu$ is upper finite.
Also note that for any $0 \leq p \leq \infty$ either upper $p$-finite or lower $p$-finite imply lower $\infty$-finite.

Let $\mu \in \pBorel$ and $0 \leq p \leq \infty$.
Say that $\mu$ is \emph{$p$-finite} if $\mu$ is lower $p$-finite and upper $p$-finite.
Since $\mu = \mu^\eps + \mu_\eps$,
for $p < \infty$, $\mu$ is $p$-finite if and only if $\mu(d_A^p) < \infty$.
That is, the indefinite integral measure~\cite[234J]{fremlin2} $d_A^p \mu$ is finite. 
In other words, $\mu$ has finite $p$-th central moment about $A$.
In particular, $\mu$ is $0$-finite if and only if $\mu$ is finite.
Furthermore, $\mu$ is $\infty$-finite if and only if for all $\eps > 0$, $\mu_\eps(X) < \infty$ and there exists $\delta>0$ such that $\mu_\delta = 0$.

Let $\mu \in \pBorel$ and let $0 \leq p \leq \infty$.
Say that $\mu$ \emph{locally $p$-finite} at $x \in X$ if there exists a neighborhood $U$ of $x$ such that $\mu_U$ is $p$-finite.
Say that $\mu$ is locally $p$-finite if it is locally $p$-finite at all $x \in X$. 
In particular, $\mu$ is locally $0$-finite if and only if $\mu$ is locally finite.

\begin{lemma} \label{lem:locally-finite}
    Let $\mu \in \pBorel$, $x \in X \setminus A$, and $0 \leq p \leq \infty$.
    Then $\mu$ is locally $p$-finite at $x$ if and only if $\mu$ is locally finite at $x$.
\end{lemma}

\begin{proof}
    First consider that case that $p=\infty$.
    Suppose that $\mu$ is locally $\infty$-finite at $x$.
    Since $x \in X \setminus A$, there is a neighborhood $U$ of $x$ such that 
    $\mu_U$ is $\infty$-finite and
    for all $y \in U$, $d_A(y) > \eps$ for some $\eps > 0$. 
    Then $\mu(U) = \mu_U(X) = (\mu_U)_\eps(X) < \infty$.
    Suppose $\mu$ is locally finite at $x$.
    Then $x$ has a neighborhood $U$ such that $\mu(U) < \infty$
    and $U \in A^\delta$ for some $\delta > 0$.
    Let $\eps > 0$.
    Then $(\mu_U)_\eps(X) \leq \mu_U(X) = \mu(U) < \infty$.
    Also, $(\mu_U)_\delta(X) = \mu_\delta(U) = 0$.
    
    Now assume that $p< \infty$.
    Suppose $\mu$ is locally $p$-finite at $x$.
    Then there is a neighborhood $V$ of $x$ such that $d_A^p\mu(V) < \infty$ and 
    for all $y \in V$, $d_A(y) \geq \eps$ for some $\eps > 0$.
    Then $\eps^p \mu(V) \leq d_A^p\mu(V) < \infty$.

    Suppose $\mu$ is locally finite at $x$.
    Then $x$ has a neighborhood $V$ such that $\mu(V) < \infty$ and for all $y \in V$, $d_A(y) \leq M$ for some $M \geq 0$. 
    Then $d_A^p \mu(V) \leq M^p \mu(V) < \infty$.
\end{proof}

\begin{lemma}
    Let $\mu \in \pBorel$, $x \in A$, and $0 \leq p \leq \infty$. 
    Then $\mu$ is locally $p$-finite at $x$ if and only if $x$ has a neighborhood $U$ such that $\mu_U$ is lower $p$-finite.
\end{lemma}

\begin{proof}
    The forward direction follows from the definitions. It remains to show the reverse direction.
    First consider the case that $p= \infty$. 
    Suppose that there is a neighborhood $U$ of $x$ such that $\mu_U$ is lower $\infty$-finite.
    Then there is a neighborhood $V$ of $x$ such that $\mu_V$ is lower $\infty$-finite and for all $y \in V$, $d_A(y) < \delta$ for some $\delta > 0$.
    We claim that $\mu_V$ is upper $\infty$-finite.
    Indeed, for all $\eps \geq \delta$, $(\mu_V)_\eps(X) = 0$ and for $0 < \eps < \delta$, $(\mu_V)_\eps(X) = (\mu_V)_\eps^\delta(X) < \infty$.

    Now assume that $p< \infty$ and that $x$ has a neighborhood $U$ such that $\mu_U$ is lower $p$-finite.
    Then there is a neighborhood $V$ of $x$ such that $\mu_V$ is lower $p$-finite and for all $y \in V$, $d_A(y)< \delta$ for some $\delta>0$.
    We claim that $\mu_V$ is upper $p$-finite.
    Indeed, for $\eps \geq \delta$, $(\mu_V)_\eps(X) = 0$ and for $0 < \eps < \delta$, $(\mu_V)_\eps(d_A^p) = (\mu_V)_\eps^\delta(d_A^p) \leq (\mu_V)^\delta(X) < \infty$.
\end{proof}

\subsection{Radon measures on metric pairs}
\label{sec:weak-radon}

Let $0 \leq p \leq \infty$.

\begin{definition} \label{def:weak-radon}
Let $\mu \in \pBorel$. 
Say that $\mu$ is a \emph{$p$-finite Radon measure on $(X,A)$} if it is tight and $p$-finite.
Say that $\mu$ is a \emph{locally $p$-finite Radon measure on $(X,A)$} if it is tight and locally $p$-finite.
Let $\pfpRadon \subset \pBorel$ denote the subset of $p$-finite Radon measures.
Let $\lpfpRadon \subset \pBorel$ denote the subset of locally $p$-finite Radon measures.
\end{definition}

For example, for $x \notin A$, the Dirac measure $\delta_x$ is a $p$-finite Radon measure.
Since $p$-finite implies locally $p$-finite, $\pfpRadon \subset \lpfpRadon$.
Since $\mu \in \lpfpRadon$ is tight it is $\tau$-additive.
For the case that $p=0$,
$\mu$ is a locally $0$-finite Radon measure on $(X,A)$ if and only if $\mu_{X \setminus A}$ is a Radon measure on $X$
and
$\mu$ is a $0$-finite Radon measure on $(X,A)$ if and only if $\mu_{X \setminus A}$ is a finite Radon measure on $X$.

Note that for $p \in (0,\infty]$ and $[\mu] \in \lpfpRadon$, there may be $x \in A$ such that every open neighborhood $U$ of $x$, $\mu(U \cap (X \setminus A)) = \infty$. So $[\mu]$ may not have a representative $\mu \in \pBorelXonly$ such that $\mu$ is Radon.

\begin{lemma} \label{lem:mu-eps-finite}
    Let $0 \leq p \leq \infty$.
    Let $\mu \in \pfpRadon$ and let $\eps > 0$.
    Then $\mu_\eps \in \fpRadon$.
\end{lemma}

\begin{proof}
    Since $\mu$ is upper $p$-finite, by \cref{lem:upper-finite-q-p}, $\mu$ is upper $0$-finite.
    Thus $\mu_\eps$ is finite.
    By the definitions, if $\mu$ is tight then so is $\mu_\eps$.
\end{proof}

\begin{remark}
    Let $0 \leq p < \infty$.
    Recall that there is a bijection between $\pBorel$ and Borel measures on $X \setminus A$
    given by $[\mu] \mapsto \mu_{X \setminus A}$.
    Since $d_A^p$ is continuous and positive on $X \setminus A$, 
    we have a bijection of Borel measures on $X \setminus A$
    given by $\mu \mapsto d_A^p \mu$ and $\nu \mapsto \frac{1}{d_A^p} \nu$.
    This bijection preserves tightness~\cite[412Q]{fremlin4}.
    By definition, $[\mu] \in \pfRadon$ is $p$-finite if and only if
    $d_A^p \mu$ is finite.
    So we have a bijection between $\pfpRadon$ and finite Radon measures on $X \setminus A$. 

    On the other hand, for $\mu \in \lpfpRadon$, $d_A^p \mu$
    is a Radon measure on $X \setminus A$. 
    However, given a Radon measure $\nu$ on $X \setminus A$,
    $[\frac{1}{d_A^p} \nu] \in \pBorel$ is tight and locally $p$-finite for all $x \in X \setminus A$, but need not be locally $p$-finite at $x \in A$. 

    Furthermore, given $\mu \in \lpfpRadon$, $d_A^p \mu$ is a Radon measure on $X$. If $\mu \in \pfpRadon$ then $d_A^p \mu$ is a finite Radon measure on $X$.
\end{remark}

\begin{lemma} \label{lem:finite-on-compact-sets}
    Let $\mu \in \lpfpRadon$ and let $K \subset X \setminus A$ be a compact set.
    Then $\mu(K) < \infty$.
\end{lemma}

\begin{proof}
    By \cref{lem:locally-finite}, each $x \in K$ has an open neighborhood $U_x$ such that $\mu(U_x) < \infty$.
    Since $K$ is compact, it has an open cover $U_{x_1},\ldots,U_{x_n}$.
    Therefore $\mu(K) < \infty$.
\end{proof}

Observe that $\lpfpRadon$ is a commutative monoid under addition, with neutral element the zero measure.
Also, there is an obvious $\R^+$ action, $a\cdot \mu = a\mu$,
for which $\lpfpRadon$ is a convex cone.
Furthermore, $\lpfpRadon$ has a elementwise partial order given by $\mu \leq \nu$ if $\mu(E) \leq \nu(E)$ for all Borel sets $E \subset X \setminus A$.
This order is compatible with addition and the $\R^+$ action, making $\lpfpRadon$ an ordered convex cone.

\begin{proposition} \label{prop:pRadon-convex-cone}
  $\lpfpRadon$ is a zero-sum-free Riesz cone.
\end{proposition}

\begin{proof}
    First, we show that the commutative monoid $\lpfpRadon$ is zero-sum-free and cancellative.
    If $\mu,\nu \in \lpfpRadon$ and $\mu + \nu =0$ then $\mu=\nu=0$.
    Let $\lambda, \mu, \nu \in \lpfpRadon$ such that $\mu + \lambda = \nu + \lambda$.
    Let $K \subset X \setminus A$ be a compact set.
    By assumption $\mu(K) + \lambda(K) = \nu(K) + \lambda(K)$.
    By \cref{lem:finite-on-compact-sets}, $\mu(K) = \nu(K)$.
    Therefore, by \cref{prop:uniqueness}, $\mu = \nu$.

  Next, we show that the elementwise partial order coincides with the natural partial order.
  That is, for $\mu,\nu \in \lpfpRadon$, $\mu(E) \leq \nu(E)$ for all Borel sets $E \subset X \setminus A$ if and only if
  there exists $\lambda \in \lpfpRadon$ such that $\mu + \lambda = \nu$.
  Let $\mu, \nu \in \lpfpRadon$.
  Assume that there exists $\lambda \in \pfpRadon$ such that $\mu + \lambda = \nu$. 
  Let $E$ be a Borel set in $X \setminus A$. 
  Then $\mu(E) + \lambda(E) = \nu(E)$. 
  Hence $\mu(E) \leq \nu(E)$.
  For the converse assume that for all Borel sets $E \subset X \setminus A$, $\mu(E) \leq \nu(E)$.
  Define the set function $\lambda$ on Borel sets $E \subset X \setminus A$
  by 
  $\lambda(E) = \sup\{\nu(E')-\mu(E')\}$, where the supremum is taken over all Borel sets $E' \subset E$ such that $\mu(E')<\infty$.
  Since $\mu,\nu$ are countably additive, so is $\lambda$.
  Therefore $\lambda \in \pBorel$ and $\mu + \lambda = \nu$.    
  Since $\mu$ and $\nu$ are locally finite, so is $\lambda$.
  Since $\nu$ is tight and $\lambda \leq \nu$, $\lambda$ is tight.
  Therefore $\lambda \in \lpfpRadon$.

   It remains to show that $\lpfpRadon$ is a lattice cone.
  Let $\mu,\nu \in \lpfpRadon$. 
  For a Borel set $E \subset X \setminus A$, we define
  $(\mu \vee \nu)(E) = \sup \{ \mu(E_1) + \nu(E_2) \}$, 
  where the supremum is taken over all partitions of $E$ into Borel sets $E_1$ and $E_2$.
  To see that $\mu \vee \nu$ is countably additive, consider a sequence $\{E_k\}_{k=1}^{\infty}$ of pairwise disjoint Borel sets.
  Observe that there is a bijection between
  partitions
  of $\bigcup_{k=1}^{\infty} E_k$
  into two disjoint Borel sets
  and disjoint partitions of each $E_k$ into two Borel sets. 
  Since $\mu$ and $\nu$ are countably additive, so is $\mu \vee \nu$.
  Next, note that $\mu \vee \nu \leq \mu + \nu$. 
  Since $\mu$ and $\nu$ are locally finite and tight, so is $\mu \vee \nu$. 
  We need to check that $\mu\vee\nu$ is indeed the supremum of $\mu$ and $\nu$. 
  If $\mu\leq \kappa$ and $\nu\leq \kappa$, then for any Borel set $E$ and for any partition of $E$ into two disjoint sets $E_1$ and $E_2$, we have $\mu(E_1) + \nu(E_2) \leq \kappa(E_1) + \kappa(E_2) = \kappa(E)$ and hence $(\mu\vee \nu)(E)\leq \kappa(E)$. Thus $\mu\vee\nu\leq \kappa$.
  Finally, for $\lambda \in \lpfpRadon$, $(\mu \vee \nu) + \lambda = (\mu+\lambda) \vee (\nu+\lambda)$ since $\lambda$ is additive.
\end{proof}

\begin{proposition} \label{prop:ideal}
    $\pfpRadon$ is an ideal in $\lpfpRadon$.
\end{proposition}

\begin{proof}
    We only treat the case $p < \infty$. The case $p = \infty$ follows similarly.
    First, we show that $\pfpRadon$ is a sub-Riesz cone.
    Let $\mu,\nu \in \pfpRadon$ and let $a \in \R^+$.
    Then $(\mu + \nu)(d_A^p) = \mu(d_A^p) + \nu(d_A^p) < \infty$ and $(a\mu)(d_A^p) = a \mu(d_A^p) < \infty$.
    Also, $0 \in \pfpRadon$.
    Furthermore, since $\mu \vee \nu \leq \mu + \nu$, $(\mu \vee \nu)(d_A^p) < \infty$.
    Finally, let $\nu \in \pfpRadon$ and $\mu \in \lpfpRadon$ with $ \mu \leq \nu$. 
    Then $\mu(d_A^p) \leq \nu(d_A^p) < \infty$.
\end{proof}


\subsection{1-finite Radon measures on metric pairs} 
\label{sec:1-finite}

Recall that the $1$-finite Radon measures on $(X,A)$
are given by
    $\pfRadon = \{ \mu \in \pBorel \ | \ \mu \text{ is tight and } \mu(d_A) < \infty\}$.
They have the following equivalent functional analytic definition.

\begin{proposition}\label{prop:measures_spaces_equiv_def_2}
    Let $\mu\in \pBorel$. 
    Then 
    $\mu(d_A) < \infty$
    if and only if $\mu(f)<\infty$ for all $f\in \LipXpos$. 
    
    Thus, 
    if $\mu$ is $1$-finite, then $\mu$ is a positive linear functional on $\LipX$. Furthermore,
\begin{equation} \label{eq:pfRadon-lip}
    \pfRadon = \{\mu\in \pBorel \ | \ \mu \textup{ is tight and } \mu(f) <\infty \textup{ for all $f\in \LipXpos$}\}.
\end{equation}
\end{proposition}

\begin{proof}
    Let $\mu \in \pBorel$.  
    In the forward direction, assume that $\mu(d_A) < \infty$.
    Let $f \in \LipXpos$. 
    Then $f \leq \lipnorm{f} d_A$. 
    Therefore $\mu(f) \leq \lipnorm{f} \mu(d_A) < \infty$.
    The reverse direction follows easily, since $d_A \in \LipXpos$.
\end{proof}


For an equivalent functional analytic definition of Radon measure, we need to assume that $(X,d)$ is locally compact.

\begin{proposition}\label{prop:measures_spaces_equiv_def_1}
Let $\mu \in \pBorel$. 
\begin{enumerate}
    \item \label{it:3.26a} If $\mu$ is locally $1$-finite then for all $f \in \LipcXpos$, $\mu(f) < \infty$.
    Thus, $\mu$ is a positive linear functional on $\LipcX$.
    \item \label{it:3.26b} If $X$ is locally compact and for all $f \in \LipcXpos$, $\mu(f) < \infty$, then $\mu$ is locally $1$-finite.
\end{enumerate}
    Thus, if $(X,d)$ is locally compact then
\begin{equation*}
  \pRadon = \{ \mu \in \pBorel \ | \ \mu \text{ is tight and }
  \mu(f) <\infty \textup{ for all $f\in \LipcXpos$}\}.
\end{equation*}
\end{proposition}

\begin{proof}
 \eqref{it:3.26a}  
    Assume that $\mu$ is locally $1$-finite.
    Then each $x \in X$ has a neighborhood $U_x$ such that $\mu_{U_x}$ is $1$-finite.
    That is $\mu_{U_x}(d_A) = (d_A \mu)(U_x) < \infty$.
    Let $f \in \LipcXpos$.
    Then $f \leq L(f)d_A$ and 
    there exists a compact set $K$ such that $\supp(f) \subset K$.
    Since $K$ is compact the cover $\{U_x\}_{x \in K}$ has a finite subcover $\{U_{x_1},\ldots,U_{x_n}\}$.
    Therefore $\mu(f) \leq L(f)\mu_K(d_A) = L(f)(d_A \mu)(K) \leq L(f) \sum_{i=1}^n (d_A \mu)(U_x) < \infty$.

\eqref{it:3.26b}
    Assume that $X$ is locally compact and that for all $f \in \LipcXpos$, $\mu(f) < \infty$.

    Let $x \in X$. 
    Since $X$ is locally compact, $x$ has a compact neighborhood $N_1$.
    Also, $x$ has a compact neighborhood $N_2$ contained in the interior of $N_1$. 
    Furthermore, there exists a Lipschitz function $f:X \to \R$ such that $f|_{N_2} = 1$ and $f|_{X \setminus N_1} = 0$. 
    Hence $d_A f \in \LipcXpos$.
    Since $\mu(d_A f)< \infty$, $d_A \mu(N_2)<\infty$, and thus $\mu$ is locally $1$-finite at $x$.
\end{proof}

\subsection{Real-valued Radon measures on metric pairs}
\label{sec:signed}

Let $0 \leq p \leq \infty$.
In this section, we consider differences of Radon measures on metric pairs.
We need to take care, since each of the two Radon measures may take the value $\infty$ for some of the Borel sets.
A key fact is that $\pfpRadon$ and $\lpfpRadon$ are cancellative commutative monoids (\cref{prop:pRadon-convex-cone}).

\begin{definition}
  Let $\pfrvRadon$ be the Grothendieck group of $\pfpRadon$
  and
  let $\lpfrvRadon$ be the Grothendieck group of $\lpfpRadon$ 
\end{definition}

Then $\pfrvRadon$ is a subspace of the vector space $\lpfrvRadon$.
Following L.\ Schwartz~\cite[p. 57]{Schwartz:1973}, 
we call elements of $\lpfrvRadon$ 
\emph{locally $p$-finite} real-valued Radon measures on $(X,A)$.
Note that for $\mu = \mu^+ - \mu^- \in \lpfrvRadon$, 
the set function 
on the Borel sets of $X \setminus A$
given by $\mu_{X \setminus A} = \mu^+_{X \setminus A} - \mu^-_{X \setminus A}$ 
is only guaranteed to be defined for relatively compact Borel sets of $X \setminus A$. However it is countably additive wherever it is defined.
We call elements of $\pfrvRadon$ \emph{$p$-finite real-valued Radon measures on $(X,A)$}.
For $\mu \in \pfrvRadon$,
$d_A\mu$ is a signed Radon measure on $X$.

Combining \cref{prop:pRadon-convex-cone,lem:Riesz-space-cone}, we have the following.

\begin{corollary}
    $\lpfrvRadon$ is a Riesz space. 
\end{corollary}

Combining \cref{prop:ideal,prop:bijection-ideal}, we have the following.

\begin{corollary}
    $\pfrvRadon$ is an ideal of $\lpfrvRadon$.
\end{corollary}

\section{Linear functionals on Lipschitz functions on metric pairs}
\label{sec:representation}

In this section we prove that certain linear functionals on $\LipcX$ and $\LipX$ can be represented as integration with respect to $1$-finite and locally $1$-finite Radon measures on $(X,A)$. 

Let $(X,d,A)$ be a metric pair denoted by $(X,A)$.
Assume that $A \neq \emptyset$.
Recall that the $1$-finite Radon measures on $(X,A)$ are given by 
\begin{equation} \label{eq:pfRadon}
    \pfRadon = \{ \mu \in \pBorel \ | \ \mu(d_A) < \infty, \ \mu \text{ is tight}\} 
\end{equation}
and that the locally $1$-finite Radon measures on $(X,A)$ are given by 
\begin{equation*}
    \pRadon = \{ \mu \in \pBorel \ | \ \forall x \in X, \exists \text{ neighborhood } U \text{ with } \mu_U(d_A) < \infty, \ \mu \text{ is tight}\}.
\end{equation*}
Let $T$ be a positive linear functional on a Riesz space $E$. 
Then $T$ is order preserving and hence order bounded.
That is, $T \in E^{\sim} = \mathcal{L}_b(E,\R) = \mathcal{L}_r(E,\R)$.
Furthermore, $T$ is sequentially order continuous (i.e. $T \in E_c^{\sim})$ if and only if $x_n \downarrow 0$ implies $Tx_n \to 0$.

\subsection{Linear functionals on compactly supported Lipschitz functions}

\begin{lemma} \label{lem:sequentially-order-continuous}
    \begin{enumerate}
        \item Let $\mu \in \pBorel$ such that $\mu$ is $1$-finite.
        Then $\mu$ is a sequentially order continuous, positive linear functional on $\LipX$.
        Thus, $\mu \in \LipX_c^{\sim}$.
        \item 
        Let $\mu \in \pBorel$ such that $\mu$ is locally $1$-finite.
        Then $\mu$ is a sequentially order continuous, positive linear functional on $\LipcX$. 
        Thus, $\mu \in \LipcX_c^{\sim}$.
    \end{enumerate}
\end{lemma}

\begin{proof}
    Let $\mu \in \pBorel$. If $\mu$ is $1$-finite then
    by \cref{prop:measures_spaces_equiv_def_2}, 
    $\mu$ is a positive linear functional on the Riesz space $\LipX$. 
    Furthermore, by 
    Beppo Levi's theorem, 
    $\mu$ is sequentially order continuous.
    Since $\mu$ is a positive linear functional, $\mu$ is order bounded.
    Thus, $\mu \in \LipX_c^{\sim}$.
    If $\mu$ is locally $1$-finite, then the result following similarly, using
    \cref{prop:measures_spaces_equiv_def_1} instead of \cref{prop:measures_spaces_equiv_def_2}.
\end{proof}

It is a result of Lozanovsky~\cite{Wulich_2016,ConesAndDuality} that for an ordered Banach space $E$ with closed and generating positive cone, $E^{\sim} \subset E'$. We give a direct proof for our case.

\begin{lemma}\label{lem:order_bounded_implies_norm_bounded}
	Let $T:\LipX\to \R$ be an order-bounded linear functional. Then $T$ is 
 a bounded linear functional. That is, $\LipX^{\sim}\subset \LipX'$.
 It follows that $T$ restricts to a bounded linear functional on $\LipcX$.
	\begin{proof}
          Suppose that $T$ is order bounded. Let $S \subset \LipX$ be norm bounded. Then there exists $M> 0$ such that $\sup_{f\in S}\lipnorm{f} \leq M <\infty$. Hence 
          $|f|\leq Md_A$ for all $f\in S$ and thus 
          $S\subset [-M d_A, M d_A]$. Since $T$ is order bounded, $T(S)$ is bounded in $\R$ in both the order-theoretic and metric senses, which are equivalent in $\R$.
Therefore $T$ is bounded.
	\end{proof}
      \end{lemma}

Since positive linear functionals are order bounded,
positive linear functionals  on $\LipX$
are bounded. 
In fact, their norm can be computed directly.

\begin{lemma}\label{lem:plf_automatically_bounded_on_Lip}
    Every positive linear functional $T: \LipX\to \R$ is bounded.
    If $A \neq \emptyset$, then
    $\|T\|_{\textup{op}} = T(d_A)$. 
\end{lemma}

\begin{proof}
    Assume $A \neq \emptyset$.
    If $A=X$ then $d_A = 0$ and $\norm{T}_{\textup{op}} = 0 = T(d_A)$. Assume $A \neq X$.
    For $f\in \LipX$ and $x\in X$ we have $|f(x)| 
    \leq 
    \lipnorm{f}d_A(x)$. 
    Hence $|T(f)| \leq T(|f|) \leq \lipnorm{f}T(d_A)$ so that 
    $\|T\|_\textup{op}\leq T(d_A)$. 
    On the other hand, $\lipnorm{d_A} = 1$ so that $T(d_A)\leq \|T\|_{\textup{op}}$. 
\end{proof}

\begin{theorem} \label{thm:representation-lipc}
    Let $(X,A)$ be a metric pair. Assume that $X$ is locally compact.
    Let $T$ be a sequentially order continuous positive linear functional on $\LipcX$.
    Then $T$ is represented by a unique 
    $\mu \in \pRadon$,
    where for each compact set $K \subset X \setminus A$, 
    $\mu(K) = \inf \{ Th \ | \ h \geq 1_K, h \in \LipcXpos\}$.
\end{theorem}


\begin{proof}
    We will apply a representation theorem of Pollard and Tops{\o}e~\cite[Theorem 3]{PollardTopsoe:1975}.
    It is easy to check that the conditions are satisfied, as follows.
    Condition A1 is satisfied since $\LipcXpos$ is a zero-sum-free Riesz cone.
    A2 is satisfied since $T$ is a linear functional and since $T$ is positive it is order preserving.
    For A3, $\LipcXpos$ is a zero-sum-free Riesz cone and for $h \in \LipcXpos$, $h \meet 1 \in \LipcXpos$.
    Let $\mathcal{K}$ be the collection of compact subsets of $X \setminus A$.
    Then $\emptyset \in \mathcal{K}$ and $\mathcal{K}$ is closed under finite unions and intersections, giving A4.
    For $h \in \LipcXpos$, $h$ is continuous, so $h^{-1}[0,a]$ is closed for all $a \geq 0$ and A5 is satisfied.
    Since $X$ is a metric space, A6$'$ is satisfied. Hence A6 is satisfied.
    $\mathcal{K}$ is closed under arbitrary intersections.
    Each $h \in \LipcXpos$ is compactly supported, so 
    the ``$\mathcal{K}$ exhausts $T$'' condition is trivially satisfied.
    For $h \in \LipcXpos$, $h \meet n \uparrow h$.
    Since $T$ is sequentially order continuous, $T(h \meet n) \uparrow Th$, and thus $(10)$ is satisfied.
    We are left with showing that $T$ is \emph{$\tau$-smooth at $\emptyset$ with respect to $\mathcal{K}$}.
    That is, if a net $K_{\alpha} \downarrow \emptyset$ in $\mathcal{K}$
    then $\inf \{Th \ | \ h \geq 1_{K_{\alpha}} \text{for some } \alpha\} = 0$.
    Consider a net $K_{\alpha} \downarrow \emptyset$ in $\mathcal{K}$.
    Choose an element $K_{\beta}$ of this net. 
    Since $X$ is Hausdorff, $\{K_{\alpha}^c\}$ is an open cover of $K_{\beta}$.
    Since $K_{\beta}$ is compact, it has a finite subcover $K_{\alpha_1}^c,\ldots,K_{\alpha_n}^c$.
    Therefore the collection $\{K_\beta,K_{\alpha_1},\ldots,K_{\alpha_n}\}$ has empty intersection.
    Since $\{K_\beta,K_{\alpha_1},\ldots,K_{\alpha_n}\}$ 
    is contained in the net $\{K_{\alpha}\}$, it follows that its intersection is as well, and thus $K_\alpha = \emptyset$ for some $\alpha$.
    So, $\inf \{Th \ | \ h \geq 1_{K_{\alpha}} \text{for some } \alpha\} = T(0) = 0$. 
    Therefore, there is unique tight, $\tau$-additive Borel measure $\mu$ on $X \setminus A$ representing $T$,
    where for each $K \in \mathcal{K}$, 
    $\mu(K) = \inf \{ Th \ | \ h \geq 1_K, \ h \in \LipcXpos\}$.
    Since $T$ is finite on $\LipcXpos$, $\mu$ is finite on $\LipcXpos$.
    By \cref{prop:measures_spaces_equiv_def_1}, $\mu \in \pRadon$.
\end{proof}

Before stating some corollaries to this result, we give an example showing there exist positive linear functionals $\LipcX$ that are not sequentially continuous and hence cannot be represented by elements of $\pfRadon$.

\begin{example}
    Let $(\R^+,0)$ denote the metric pair $(\R^+,d,\{0\})$, where $d(a,b) = \abs{a-b}$.
    Here we show that there exists a positive linear functional on $\Lip_c(\R^+,0)$ 
    that is not sequentially order continuous. 
    For $k \in \Z$, 
    let $g_k = (d_{0} \meet (2^k - d_{0})) \join 0$.
    Then $B = \{g_k\}_{k \in \Z}$ 
    is a linearly independent subset of $\Lip_c(\R^+,0)$. 
    Let $W$ denote the subspace of $\Lip_c(\R^+,0)$ generated by $B$. 
    Then $W$ has the property that for all $f\in \Lip_c(\R^+,0)$, there exists $g\in W$ with $g\geq f$.
    Define $T:W\to \R$ by setting $T(g) = 1$ for all $g\in B$ and then extending linearly to $W$. 
    Consider $f = \sum_{j=1}^m c_j g_{k_j}$ in $W$.
    Then $T(f) = \sum_{j=1}^m c_j$,
    and for $x$ sufficiently close to $0$, 
    $f(x) = \sum_{j=1}^m c_j x =  x T(f)$. 
    Hence if $f \geq 0$ then $T(f)\geq 0$. 
    That is, $T$ is positive. 
    By \cref{prop:positive_hahn_banach_weak}, $T$ extends to a positive linear functional $\tilde{T}:\Lip_c(\R^+,0)\to \R$.
    As $k \to -\infty$, $g_k \downarrow 0$ 
    but $\tilde{T}(g_k) = T(g_k) = 1$, and hence $\tilde{T}$ is not sequentially order continuous.
\end{example}


\begin{corollary} \label{cor:cpt_lip_determines_pRadon}
    Assume that $X$ is locally compact. 
    Let $\mu,\nu \in \pRadon$. Then $\mu=\nu$ if and only if $\mu(f)=\nu(f)$ for all $f \in \LipcXpos$.
\end{corollary}


\begin{corollary}\label{cor:measure_spaces_equiv_def_3}
    Assume that $X$ is locally compact.
    Then \[\pRadon = \{ \mu \in \pBorel \ | \ \forall f \in \LipcXpos,\  \mu(f) < \infty \}.\]
\end{corollary}

\begin{proof}
    Let $\mu \in \pBorel$ such that for all $f \in \LipcXpos$, $\mu(f) < \infty$.
    Since integration is linear, $\mu$ is a positive linear functional on $\LipcX$.
    By Beppo Levi's lemma, $\mu$ is sequentially order continuous.
    By \cref{thm:representation-lipc}, $\mu \in \pRadon$.
    The reverse direction is given by \cref{prop:measures_spaces_equiv_def_1}.
\end{proof}

\begin{remark}
    If we only assume that $X \setminus A$ is locally compact then \cref{thm:representation-lipc} does not hold.
    Consider the following example arising in persistent homology.
    Let $Y = \{(x,y) \in \R^2 \ | \ x \leq y\}$ with the Euclidean metric and 
    let $\Delta = \{(x,x) \ | \ x \in \R\}$.
    Let $(X,d)$ be the quotient metric space $Y/\Delta$ and consider the metric pair $(X,d,A)$ where $A$ is the one-point set containing the equivalence class $\Delta$.
    Then $X \setminus A$ is locally compact but $X$ is not locally compact. 
    Thus, for all $f \in \LipcXpos$, $f$ vanishes in a neighborhood of the point $\Delta$.
    Let $\mu = \sum_{k=1}^{\infty} \delta_{(0,\frac{1}{k})}$.
    Then $\mu \in \pBorel$ and for all $f \in \LipcXpos$, $\mu(f) < \infty$. 
    However, $\mu$ is not locally $1$-finite at $A$,
    so $\mu \not\in \pRadon$.
\end{remark}

\begin{theorem} \label{thm:representation-lipc-non-positive}
    Let $(X,A)$ be a metric pair. 
    Assume that $X$ is locally compact.
    For any order bounded, sequentially order continuous linear functional $T:\LipcX\to \R$, there exists $\mu,\nu \in \pRadon$ such that $T(f) = \int_X f d(\mu-\nu)$ for all $f\in \LipcX$. 
    Moreover, $\mu$ and $\nu$ can be chosen uniquely such that, for all $f\in \LipcXpos$, 
	\[\textstyle \inf\{\int_X g d\mu + \int_X h d\nu \ | \ g+ h = f, \ g,h \in \LipcXpos\} = 0.\] 
    That is, $\LipcX_c^{\sim} = \Radon$.
\end{theorem}
	
\begin{proof}
    Since $T$ is order bounded, 
    it is an element of the order dual $\LipcX^{\sim}$ of the Riesz space $\LipcX$.
    Hence, it has a unique decomposition
    $T = T^+ - T^-$, where $T^+$ and $T^-$ are positive linear operators with $T^+\wedge T^- = 0$.
    Since $T$ is sequentially order continuous, so are $T^+$ and $T^-$. By Theorem \ref{thm:representation-lipc}, there exists unique $\mu,\nu\in \pRadon$ such that $T^+(f) = \int_Xf d\mu$ and $T^-(f) = \int_X f d\nu$ for all $f\in \LipcX$. Hence $T(f) = \int_X fd(\mu-\nu)$ for all $f\in \LipcX$. 
    The uniqueness statement is simply a restatement of the uniqueness of the decomposition $T = T^+ - T^-$ with $T^+\wedge T^- = 0$, expressed in terms of $\mu$ and $\nu$.
\end{proof}

\subsection{Linear functionals on Lipschitz functions}

\begin{definition} \label{def:exhausted}
    Let $T: \LipX \to \R$ be a positive linear functional.
    Say that $T$ is \emph{exhausted by compact sets} 
    if for all $f \in \LipXpos$ and for all $\eps > 0$,
    there exists a compact set $K \subset X \setminus A$ such that 
    $\sup \{ T(g) \ | \ g|_K = 0, \ g \leq f, \ g \in \LipXpos \} < \eps$.
    For $T \in \LipX^\sim$, say that $T$ exhausted by compact sets if $\abs{T}$ is exhausted by compact sets.
\end{definition}

\begin{lemma} \label{lem:exhausted}
    Let $T: \LipX \to \R$ be a positive linear functional. 
    Then $T$ is exhausted by compact sets  if and only if
    for all $\eps > 0$
    there is a compact set $K \subset X \setminus A$ such that
    for all $L> 0$, $T(d_A\meet Ld_K) <\eps$.
\end{lemma}

\begin{proof}
    In the forward direction, let 
    $f = d_A$ and let 
    $g = d_A\meet Ld_K$.
    In the reverse direction, 
     fix $f \in \LipXpos$. Then $f \leq L(f)d_A$ and for $g \in \LipXpos$ with $g \leq f$ and $g|_K = 0$ for some $K \subset X \setminus A$, $g \leq L(f)d_A \meet L(g)d_K$.
    For $\eps > 0$, there is a compact set $K \subset X \setminus A$ such that for all $L > 0$,
    $T(d_A \meet L d_K) < \frac{\eps}{L(f)}$.
    Then for $g \in \LipXpos$ with $g \leq f$ and $g|_K = 0$, 
    $T(g) \leq T(L(f)d_A \meet L(g)d_K) = L(f)T(d_A \meet \frac{L(g)}{L(f)} d_K) < \eps$.
\end{proof}

\begin{lemma}\label{lem:pfRadon_exhausted_by_cpt_sets}
    $\mu \in \pfRadon$ is exhausted by compact sets.
\end{lemma}

\begin{proof}
    Let $\mu \in \pfRadon$. 
    Let $f \in \LipXpos$.
    Then $\mu(f) < \infty$.
    Let $\eps>0$.
    By the definition of the integral, there is a simple function $\sum_{i=1}^n a_i \chi_{E_i} \leq f$, with each $E_i$ a Borel set, such that $\sum_{i=1}^n a_i \mu(E_i) > \mu(f) - \frac{\eps}{2}$.
    Since $\mu$ is tight, 
    for each $i$ there is a compact set $K_i \subset E_i$ with $\mu(K_i) > \mu(E_i) - \frac{\eps}{2n}$.
    Thus, we have the simple function $\sum_{i=1}^n a_i \chi_{K_i} \leq f$
    with $\sum_{i=1}^n a_i \mu(K_i) > \mu(f) - \eps$.
    Let $K = \bigcup_{i=1}^n K_i$.
    Choose $g \in \LipXpos$ with $g \leq f$ and $g|_K = 0$.
    Then $f-g \geq (f-g)\chi_K = f \chi_K \geq \sum_{i=1}^n a_i \chi_{K_i}$.
    Thus $\mu(f) - \eps < \sum_{i=1}^n a_i \mu(K_i) \leq \mu(f-g) = \mu(f) - \mu(g)$.
    Therefore $\mu(g) < \eps$.
\end{proof}

\begin{lemma} \label{lem:loc-compact-sigma-compact}
    If $X \setminus A$ is locally compact and $\sigma$-compact then every sequentially order continuous positive linear functional $T: \LipX \to \R$ is exhausted by compact sets.
\end{lemma}

\begin{proof}
    Since $X \setminus A$ is locally compact and $\sigma$-compact, 
    $X \setminus A$ is exhausted by compact sets.
    That is, there is a sequence $(K_n)$ 
    of compact sets in $X \setminus A$
    such that for each $n$, $K_n$ is contained in the interior of $K_{n+1}$ and $X \setminus A \subset \bigcup_{n=1}^{\infty} K_n$.
    Let $T: \LipX \to \R$ be a sequentially order continuous positive linear functional.
    Let $f \in \LipXpos$.
    Let $a_n = \inf \{ T(g) \ | \ g \in \LipXpos, \ g|_{K_n} = f, \ g|_{K_{n+1}^c} = 0\}$.
    Let $b_n = \sup \{ T(h) \ | \ h \in \LipXpos, \ h|_{K_{n+1}} = 0, \ h \leq f\}$.
    Then for all $n$, $a_n + b_n \leq T(f)$.
    Since $T$ is sequentially order continuous, $a_n \uparrow T(f)$. 
    Therefore $b_n \downarrow 0$.
    Thus, $T$ is exhausted by compact sets.
\end{proof}

Combining \cref{lem:sequentially-order-continuous,lem:pfRadon_exhausted_by_cpt_sets}, we have the following.

\begin{proposition} \label{prop:pfRadon-sequentially-order-continuous-exhausted}
    Let $\mu \in \pfRadon$. Then $\mu$ is a sequentially order continuous positive linear functional on $\LipX$ that is exhausted by compact sets.
\end{proposition}

The following result gives a converse to \cref{prop:pfRadon-sequentially-order-continuous-exhausted}.

\begin{theorem} \label{thm:representation-lip}
    Let $(X,A)$ be a metric pair.
    Let $T$ be a sequentially order continuous positive linear functional on $\LipX$.
    Then $T$ is represented by a unique 
    $\mu \in \pfRadon$
    if and only if $T$ is exhausted by compact sets.
    If so, then for each compact set $K \subset X \setminus A$, 
    $\mu(K) = \inf \{ Th \ | \ h \geq 1_K, \ h \in \LipXpos\}$.
\end{theorem}

\begin{proof}
    We will again apply the representation theorem of Pollard and Tops{\o}e~\cite[Theorem 3]{PollardTopsoe:1975}.
    Let $\mathcal{K}$ denote the compact subsets of $X \setminus A$.
    By the identical arguments as in the proof of \cref{thm:representation-lipc}, 
    A1-A5, A6$'$ and hence A6 hold, as well as (10), and also $T$ is $\tau$-smooth at $\emptyset$ with respect to $\mathcal{K}$.
    Therefore, there is unique tight, $\tau$-additive Borel measure $\mu$ on $X \setminus A$ representing $T$,
    if and only if $T$ is exhausted by compact sets, and if so, 
    for each $K \in \mathcal{K}$, 
    $\mu(K) = \inf \{ Th \ | \ h \geq 1_K, h \in \LipcXpos\}$.
    Since $T(d_A) < \infty$, if there exists such a $\mu$, then $\mu \in \pfRadon$.
\end{proof}

Before stating some corollaries to this result, we give an example of a positive linear functional on $\LipX$ that is not sequentially order continuous and hence cannot be represented by an element of $\pfRadon$.

\begin{example}
    There exists a positive linear functional on $\Lip(\R^+,0)$ 
    that is not sequentially order continuous.
    Let $M\subset \Lip(\R^+,0)$ be the linear subspace generated by $\Lip_c(\R^+,0)$ and $d_{0}$, and define $T:M\to \R$ by setting $T(g) = 0$ for all $g\in \Lip_c(\R^+,0)$ and $T(d_{0}) = 1$ and then extending linearly. Since $d_{0}$ cannot be written as a finite linear combination of compactly supported functions, $T$ is well-defined. 
    By \cref{prop:positive_hahn_banach_weak}, $T$ has an extension to a positive linear functional $\tilde{T}:\Lip(\R^+,0)\to \R$.
    For $n\geq 1$, let $g_n = (d_{0} \meet (n-d_{0})) \join 0$.
    Since $g_n \in \Lip_c(\R^+,0)$, $\tilde{T}(g_n) = 0$ for all $n$. 
    However, $g_n\uparrow d_{0}$ 
    and $\tilde{T}(d_0)=1$.
    Thus $\tilde{T}$ is not sequentially order continuous. 
\end{example}

\begin{corollary} \label{cor:lip_determines_pfRadon} 
    Let $\mu,\nu\in \pfRadon$. Then $\mu = \nu$ if and only if $\mu(f) = \nu(f)$ for all $f\in \LipXpos$.
\end{corollary}

\begin{corollary} \label{cor:representation-lip}
    \begin{enumerate}
        \item \label{it:exhausted}
        Let $\mu \in \pBorel$ such that 
        $\mu(d_A) < \infty$.
        Then $\mu$ is exhausted by compact sets if and only if $\mu$ is tight.
        \item \label{it:locally-compact-sigma-compact}
        If $X \setminus A$ is locally compact and $\sigma$-compact then
        $\pfRadon 
        = \{ \mu \in \pBorel \ | \ \mu(d_A) < \infty\}
        = \{ \mu \in \pBorel \ | \ \forall f \in \LipXpos, \ \mu(f) < \infty\}$.
    \end{enumerate}
\end{corollary}

\begin{proof}
    \eqref{it:exhausted}
    For the reverse direction, if $\mu$ is tight then by \cref{def:weak-radon}, $\mu \in \pfRadon$, and so by \cref{lem:pfRadon_exhausted_by_cpt_sets}, $\mu$ is exhausted by compact sets.
    For the forward direction, by 
    Beppo Levi's theorem, 
    $\mu$ is a sequentially order continuous positive linear functional on $\LipX$.
    Thus, by \cref{thm:representation-lip}, $\mu$ is tight.

    \eqref{it:locally-compact-sigma-compact}
    Let $\mu \in \pBorel$ such that $\mu$ is $1$-finite. 
    By \cref{lem:sequentially-order-continuous}, $\mu$ is a sequentially order continuous positive linear functional on $\LipX$.
    Thus, by \cref{lem:loc-compact-sigma-compact}, $\mu$ is exhausted by compact sets, and 
    hence by \eqref{it:exhausted}, $\mu$ is tight.
    Therefore, we may omit the tightness condition from \eqref{eq:pfRadon} and \eqref{eq:pfRadon-lip}.
\end{proof}

\begin{theorem}\label{thm:representation-lip-non-positive}
    Let $T:\LipX\to \R$ be an order bounded, sequentially order continuous linear functional,
    which is exhausted by compact sets.
    Then there exists measures $\mu,\nu\in \pfRadon$ such that $T(f) = \int_Xfd(\mu-\nu)$ for all $f\in \LipX$. 
    Moreover, $\mu$ and $\nu$ can be chosen uniquely such that, for all $f\in \LipXpos$, 
	\[\textstyle \inf\{\int_X g d\mu + \int_X h d\nu \ | \ g+ h = f, \ g,h \in \LipXpos\} = 0.\] 
\end{theorem}

\begin{proof}
    Since $T$ is order bounded, it is an element of the order dual, $\LipX^{\sim}$, of the Riesz space $\LipX$.
    Hence there is a unique decomposition $T = T^+ - T^-$, where $T^+$ and $T^-$ are positive linear functionals with $T^+\wedge T^- = 0$.
    By assumption, $T^+$ and $T^-$ are exhausted by compact sets.
    
    Since $T$ is sequentially order continuous, so are $T^+$ and $T^-$. 
    Thus, by Theorem \ref{thm:representation-lip}, there exists measures $\mu,\nu\in \pfRadon$ such that $T(f) = \mu(f) - \nu(f)$ for all $f\in \LipX$. 
    The uniqueness statement is a restatement of the uniqueness of the decomposition $T = T^+ - T^-$ with $T^+\wedge T^- = 0$. 
\end{proof}

Combining \cref{lem:loc-compact-sigma-compact,thm:representation-lip-non-positive}, we have the following

\begin{corollary} \label{cor:representation-lip-non-positive}
    If $X \setminus A$ is locally compact and $\sigma$-compact then 
    $\LipX_c^{\sim} = \fRadon$. 
\end{corollary}

\begin{lemma} \label{lem:compact-offset}
    Assume that $X \setminus A$ is locally compact. 
    Let $K$ be a compact subset of $X \setminus A$.
    Then there exists $\delta > 0$ such that $K^\delta$ is a compact subset of $X \setminus A$.
\end{lemma}

\begin{proof}
    Since $K$ is a compact subset of $X \setminus A$ and $X \setminus A$ is locally compact,
    for $x \in K$, we may choose $0 < \delta_x < d_A(x)$ such that $\overline{B_{\delta_x}(x)}$ is compact.
    Consider $\{B_{\frac{\delta_x}{2}}(x)\}_{x \in K}$.
    Since $K$ is compact, there is a finite subcover 
    $B_{\frac{\delta_{x_1}}{2}}(x_1), \ldots,
    B_{\frac{\delta_{x_n}}{2}}(x_n)$ of $K$.
    That is, for all $x \in K$ there exists $1 \leq i \leq n$ such that
    \begin{equation} \label{eq:finite-cover}
        d(x,x_i) \leq \frac{\delta_{x_i}}{2}.
    \end{equation}
    Let $K' = \overline{B_{\delta_{x_1}}(x_1)} \cup \cdots \cup
    \overline{B_{\delta_{x_n}}(x_n)}$.
    Then $K'$ is a compact subset of $X \setminus A$.
    Let $\delta = \min\{\frac{\delta_{x_1}}{2},\ldots,\frac{\delta_{x_n}}{2}\}$.
    We will show that $K^\delta \subset K'$.
    Let $x \in K^\delta$.
    That is, there exists $x' \in K$ such that $d(x,x') \leq \delta$.
    Apply \eqref{eq:finite-cover} to $x' \in K$ to choose $1 \leq i \leq n$.
    Then $d(x,x_i) \leq d(x,x') + d(x',x_i) \leq \frac{\delta_{x_i}}{2} + \frac{\delta_{x_i}}{2} = \delta_{x_i}$.
    Therefore $x \in K'$. Thus $K^\delta \subset K'$ as claimed.
    Hence $K^\delta$ is a compact subset of $X \setminus A$.
\end{proof}

\begin{lemma} \label{lem:exhausted-implies-seq-order-cont}
    Assume that $X \setminus A$ is locally compact.
    Let $T$ be a positive linear functional on $\LipX$, which is exhausted by compact sets.
    Then $T$ is sequentially order continuous.
\end{lemma}

\begin{proof}
    Let $(h_n) \subset \LipXpos$ such that $h_n \downarrow 0$.
    Let $\eps > 0$.
    Since $T$ is exhausted by compact sets, 
    there exists a compact set $K \subset X \setminus A$ such that
    $\sup\{ T(g) \ | \ g \in \LipXpos, g|_K = 0, g \leq h_1\} < \frac{\eps}{2}$.
    By \cref{lem:compact-offset}, there exists $\delta > 0$ such that $K^\delta$ is a compact subset of $X \setminus A$.

    For each $n$, let $a_n = \sup_{x \in K} h_n(x)$.
    Since $K$ is compact, by Dini's theorem $a_n \downarrow 0$.
    Let $\tilde{g}_n = \frac{a_n}{\delta} d_{(K^\delta)^c} \meet a_n \in \LipcXpos$.
    Let $g_n = \tilde{g}_n \meet h_n$ and let $f_n = h_n - g_n$.
    Then $f_n \in \LipXpos$, $f_n|_K = 0$, $f_n|_{(K^\delta)^c} = h_n|_{(K^\delta)^c}$ 
    and $f_n \leq h_n \leq h_1$.
    Therefore $T(f_n) < \frac{\eps}{2}$.

    Furthermore, $\tilde{g}_n = a_n(\frac{1}{\delta} d_{(K^\delta)^c} \meet 1)$.
    Thus $T(\tilde{g}_n) = a_n T(\frac{1}{\delta} d_{(K^\delta)^c} \meet 1) \downarrow 0$.
    Since $g_n \leq \tilde{g}_n$, $T(g_n) \leq T(\tilde{g}_n)$.
    Hence, $T(g_n) < \frac{\eps}{2}$ for $n$ sufficiently large.
    Therefore, $T(h_n) = T(f_n) + T(g_n) < \eps$ for $n$ sufficiently large.
\end{proof}

Combining \cref{lem:exhausted-implies-seq-order-cont,thm:representation-lip}, we have the following.

\begin{theorem} \label{thm:representation-lip-exhausted}
    Assume that $X \setminus A$ is locally compact.
    Let $T: \LipX \to \R$ be a positive linear functional 
    which is
    exhausted by compact sets.
    Then $T$ is represented by a unique $\mu \in \pfRadon$.
    Furthermore, 
    for each compact set $K \subset X \setminus A$, 
    $\mu(K) = \inf \{ Th \ | \ h \geq 1_K, h \in \LipXpos\}$.
\end{theorem}

Combining \cref{lem:exhausted-implies-seq-order-cont,thm:representation-lip-non-positive}, we have the following.

\begin{theorem} \label{thm:representation-lip-nonpositive-exhausted}
    Assume that $X \setminus A$ is locally compact.
    Let $T:\LipX\to \R$ be an order bounded linear functional
    which is
    exhausted by compact sets. 
    Then there exist $\mu,\nu\in \pfRadon$ such that $T(f) = \int_Xfd(\mu-\nu)$ for all $f\in \LipX$. Moreover, $\mu$ and $\nu$ can be chosen uniquely such that, for all $f\in \LipXpos$, 
	\[\textstyle \inf\{\int_X g d\mu + \int_X h d\nu \ | \ g+ h = f, \ g,h \in \LipXpos\} = 0.\] 
\end{theorem}

\subsection{Bounded linear functionals on Lipschitz functions}

We introduce the following condition 
on a metric pair $(X,A)$ which is slightly more general than $X$ being boundedly compact.

\begin{definition}
    Say that the metric pair is $(X,A)$ is boundedly compact if $X$ is locally compact and $\sigma$-compact and for each $x \in X \setminus A$ and for each $\eps, r > 0$, $\overline{B}_r(x) \cap \overline{A_\eps}$ is compact.
\end{definition}


\begin{theorem} \label{thm:representation-proper}
    Assume that 
    $(X,A)$
    is 
    boundedly compact.
    Let $T$ be a sequentially order continuous, 
    bounded,
    positive linear functional on $\LipcX$.
    Then $T$ is represented by a unique $\mu \in \pfRadon$.
\end{theorem} 

\begin{proof}
    By \cref{thm:representation-lipc}, 
    $T$ is represented by a unique $\mu \in \pRadon$.
    Fix $x\in X\backslash A$. 
    For each $n\geq 1$, let $K_n = \overline{B}_n(x)\cap \overline{A_{1/n}}$ and 
    let $h_n = d_{(K_n)^c}$. 
    Since $A \subset (K_n)^c$, $d_{(K_n)^c} \leq d_A$. Also,
    since $(X,A)$ is boundedly compact, $K_n$ is compact, and hence $h_n\in \LipcOneXpos$. 
    Moreover, $h_n \uparrow d_A$ (here, we are using the convention that $d_{(K_n)^c} = \infty$ if $(K_n)^c = \emptyset$).
    
    Since $T$ is bounded, there is an $M > 0$ such that $T(h_n) \leq M L(h_n) \leq  M$.
    Since $h_n \uparrow d_A$, 
    by Beppo Levi's lemma,
    $\mu(d_A) = \sup \mu(h_n) \leq M$.
    Therefore $\mu \in \pfRadon$.
\end{proof}

\begin{corollary}
    Assume that 
    $(X,A)$
    is boundedly compact.
    Then there is a bijection between sequentially order continuous, positive, bounded linear functionals on $\LipcX$ and sequentially order continuous positive linear functionals on $\LipX$.
\end{corollary}

\begin{proof}
  Let $T: \LipcX \to \R$ be a sequentially order continuous, positive, bounded linear functional. Then by Theorem~\ref{thm:representation-proper}, $T$ is represented by a  unique $\mu \in \pfRadon$, which gives a sequentially order continuous positive linear extension of $T$ to $\LipX$.

  Let $T$ be a sequentially order continuous positive linear functional on $\LipX$.
  By \cref{lem:plf_automatically_bounded_on_Lip},
  $T$ is bounded.
  Therefore $T$ restricts to a sequentially order continuous positive, bounded linear functional on $\LipcX$.
\end{proof}

\begin{theorem} \label{thm:representation-proper-not-positive}
    Assume that 
    $(X,A)$
    is boundedly compact.
    Let $T$ be an order bounded, sequentially order continuous, linear functional $T:\LipcX\to \R$ such that both $T$ and $\abs{T}$ are bounded.
    Then there exists $\mu,\nu \in \pfRadon$ such that $T(f) = \int_X f d(\mu-\nu)$ for all $f\in \LipcX$. Moreover, $\mu$ and $\nu$ can be chosen uniquely such that, for all $f\in \LipcXpos$, 
	\[\textstyle \inf\{\int_X g d\mu + \int_X h d\nu \ | \ g+ h = f, \ g,h \in \LipcXpos\} = 0.\] 
\end{theorem}
	
\begin{proof}
    Since $T$ is order bounded, 
    it is an element of the order dual $\LipcX^{\sim}$ of the Riesz space $\LipcX$.
    Hence, $T$ has a unique decomposition
    $T = T^+ - T^-$, where $T^+$ and $T^-$ are
    positive linear operators with $T^+\wedge T^- = 0$. 
    Since $T$ is sequentially order continuous, so are $T^+$ and $T^-$. 
    Since $T$ and $\abs{T}$ are bounded, so are
    $T^+ = \frac{1}{2}(T + \abs{T})$ and $T^- = \frac{1}{2}(T - \abs{T})$.
    By Theorem \ref{thm:representation-proper}, there exists unique $\mu,\nu\in \pfRadon$ such that $T^+(f) = \int_Xf d\mu$ and $T^-(f) = \int_X f d\nu$ for all $f\in \LipcX$. 
    Hence $T(f) = \int_X fd(\mu-\nu)$ for all $f\in \LipcX$. 
\end{proof}

\section{Relative optimal transport}
\label{sec:OT}
        
Classical optimal transport is concerned with finding the most cost effective plan for transporting one configuration of mass to another. In the classical formulation, the initial and final states must have the same finite total mass. In the relative transport problem, we 
have a reservoir which provides an unlimited source or sink for mass.
As with the classical transport problem, the relative transport problem induces a family of distances between measures called Wasserstein distances. However, unlike the classical problem, the corresponding relative  distance is well-defined between measures of different total mass.
In this section assume that all metric spaces are
complete and separable.

\subsection{Products of metric pairs}

We start with some elementary results on products of metric pairs.

Consider metric pairs $(X,d,A)$ and $(Y,e,B)$.
Assume that $A,B \neq \emptyset$.
We have the product $(X \times Y, d + e, A \times B)$.
For simplicity, denote these metric pairs $(X,A)$, $(Y,B)$ and $(X \times Y, A \times B)$.
Denote the projection maps by 
$p_1: (X \times Y, A \times B) \to (X,A)$
and
$p_2: (X \times Y, A \times B) \to (Y,B)$.
It is easy to verify the following.

\begin{lemma} \label{lem:products}
    \begin{enumerate}
        \item \label{it:projection-lipschitz} $p_1$ and $p_2$ are $1$-Lipschitz.
        \item \label{it:induced-map} 
            A morphism of metric pairs $\varphi:(X,A) \to (Y,B)$ induces
            $\varphi_*: \pBorel \to \pBorelY$.
        \item \label{it:dAeB} $(d+e)_{A \times B} = d_A \oplus e_B$.
        \item \label{it:pi-f-oplus-g}
            Let $\pi \in \pBorelXY$, $f \in \LipX$, and $g \in \LipY$.
            Then $\pi(f \oplus g) = ((p_1)_* \pi)(f) + ((p_2)_* \pi)(g)$ if either the left hand side or right hand side is defined in $[-\infty,\infty]$.
    \end{enumerate}
\end{lemma}

\begin{proof}
    \eqref{it:projection-lipschitz} and \eqref{it:dAeB} are elementary calculations.
    
    \eqref{it:induced-map}
    For a metric pair $(X,A)$, let $\iota_X$ 
    denote the canonical map from $\pBorelAonly$ to $\pBorelXonly$ by $\iota_X$.
    From the definitions, $\varphi_* \circ \iota_X = \iota_Y \circ \varphi_*$. 
    Therefore there is a canonical induced map between the quotient monoids.

    \eqref{it:pi-f-oplus-g} 
    By definition and \cite[235G]{fremlin2},
    $\pi(f \oplus g) = \pi(f \circ p_1 + g \circ p_2) = \int f p_1 d\pi + \int g p_2 d\pi
    = \int f d((p_1)_* \pi) + \int g d((p_2)_* \pi) = (p_1)_* \pi (f) + (p_2)_* \pi (g)$.
\end{proof}


In the case that $(X,d,A) = (Y,e,B)$,
we have the metric pair $(X \times X,d+d,A \times A)$ 
which we denote by $(X^2,A^2)$.
Then $d: (X^2,d+d) \to \R$ is $1$-Lipschitz but $d|_{A^2} \neq 0$ in general.
Let 
\begin{equation} \label{eq:d-bar}
    \bar{d} = d \meet (d_A \oplus d_A),
\end{equation}
Since $d$ and $(d+d)_{A \times A}$ are $1$-Lipschitz, $\bar{d} \in \LipOneXXpos$.
Also, $\bar{d}$ is a pseudometric on $X$~\cite[Lemma 3.13]{bubenik2022universality} such that $d(x,y) = 0$ if and only if either $x=y$ or $x,y \in A$.
Furthermore $\bar{d}$ is the quotient metric on $X/A$~\cite[Lemma 3.17]{bubenik2022universality}.

\begin{lemma} \label{lem:lipschitz-dbar}
    Let $f \in \LipX$.
    Then for all $x,y \in X$,
    $f(x) - f(y) \leq L(f) \bar{d}(x,y)$.
\end{lemma}

\begin{proof}
    First, $d$ is $1$-Lipschitz.
    Second, for all $x,y \in X$, $f(x) - f(y) \leq \abs{f(x)} + \abs{f(y)} \leq L(f)(d_A \oplus d_A)(x,y)$.
\end{proof}

\subsection{1-Wasserstein distance for metric pairs}

Let $(X,d,A)$ be a metric pair with 
$(X,d)$ complete and separable and
$A \neq \emptyset$.
In this section we define a function $W_1: \pBorel \times \pBorel \to [0,\infty]$,
which we call the (relative) $1$-Wasserstein distance. 
For $\mu,\nu \in \pBorel$, 
we prove that $W_1(\mu,\nu) = W_1(\nu,\mu)$, $W_1(\mu,\mu) = 0$, and that if $\mu,\nu$ are $1$-finite then $W_1(\mu,\nu) < \infty$. 
Under the additional assumption that $\mu,\nu \in \pfRadon$, we prove that $W_1(\mu,\nu) = 0$ implies that $\mu=\nu$ and that $W_1$ satisfies the triangle inequality.

Given $a \in A$, we have maps
$i_1^a: X \to X^2$ and $i_2^a:X \to X^2$ given by $i_1^a(x) = (x,a)$ and $i_2^a(x) = (a,x)$, which induce $1$-Lipschitz morphisms $i_1^a:(X,A) \to (X^2,A^2)$ and $i_2^a:(X,A) \to (X^2,A^2)$.

Throughout this section $\mu,\nu \in \pBorel$. Also $\mu_j,\nu_j \in \pBorel$ for $j=1,2$.

\begin{definition}
    A \emph{coupling} of $\mu$ and $\nu$ is a given by $\pi \in \pBorelXX$ such that
    $(p_1)_*(\pi) = \mu$ and $(p_2)_*(\pi) = \nu$.
    Let $\Pi(\mu,\nu)$ denote the set of couplings of $\mu$ and $\nu$.
\end{definition}

Note that if $\pi_1 \in \Pi(\mu_1,\nu_1)$ and $\pi_2 \in \Pi(\mu_2,\nu_2)$ then $\pi_1 + \pi_2 \in \Pi(\mu_1+\mu_2,\nu_1+\nu_2)$.

\begin{example} \label{ex:trivial-coupling}
Let $a \in A$.
    By \cref{lem:products}\eqref{it:induced-map}, we can define $\pi = (i_1^a)_* \mu + (i_2^a)_* \nu \in \pBorelXX$.
    Then $(p_1)_* \pi = (p_1 i_1^a)_* \mu + (p_1 i_2^a)_* \nu = \mu$
    and similarly $(p_2)_* \pi = \nu$.
    Call $\pi$ a \emph{trivial coupling} of $\mu$ and $\nu$.
\end{example}

\begin{example} \label{ex:trivial-extension}
    Given $\mu = \mu_1 + \mu_2$ and $\nu = \nu_1 + \nu_2$.
    Then $\pi_1 \in \Pi(\mu_1,\mu_2)$ can be \emph{trivially extended} to $\pi \in \Pi(\mu,\nu)$ by adding the trivial coupling of $\mu_2$ and $\nu_2$ in \cref{ex:trivial-coupling} to $\pi$.
\end{example}

Because of the existence of trivial couplings, we have the following.

\begin{lemma} \label{lem:couplings-nonempty}
    $\Pi(\mu,\nu) \neq \emptyset$.
\end{lemma}

\begin{example} \label{ex:diagonal-coupling}
    For $\mu \in \pBorel$, we have the \emph{diagonal coupling} $\Delta_* \mu$, where $\Delta:X \to X \times X$ is given by $x \mapsto (x,x)$.
    For $j=1,2$, $p_j \circ \Delta$ equals the identity map on $X$, and thus $(p_j)_* \Delta_* \mu = (p_j \circ \Delta)_* \mu = \mu$.
    Hence $\Delta_* \mu \in \Pi(\mu,\mu)$.
\end{example}

\begin{example} \label{ex:approximation-by-finite-measure}
    Let $\eps \geq 0$.
    Recall that $\mu = \mu^\eps + \mu_\eps$.
    Let $a \in A$.
    Combining 
    a trivial coupling of $\mu^\eps$ and $0$ and the diagonal coupling on $\mu_\eps$,
    we have that 
    $(i_1^a)_*(\mu^\eps) + \Delta_*(\mu_\eps) \in \Pi(\mu,\mu_\eps)$.
\end{example}

Recall \eqref{eq:d-bar}, $\bar{d} = d \wedge (d_A \oplus d_A) = d \wedge (d+d)_{A \times A}$.
That is, $\bar{d}(x,y) = d(x,y) \wedge (d_A(x) + d_A(y))$.
Given $a \in A$, for all $x \in X$, $\bar{d}(x,a) = d_A(x)$.

\begin{definition}
    Define the \emph{(relative) $1$-Wasserstein distance} between $\mu$ and $\nu$ to be given by 
    \begin{equation*}
        W_1(\mu,\nu)  
        = \inf_{\pi \in \Pi(\mu,\nu)} \pi( \bar{d} ) 
    \end{equation*}
\end{definition}


If $\pi\in \Pi(\mu,\nu)$ then by \cref{lem:products}\eqref{it:pi-f-oplus-g}, we have $\pi(f\oplus g) = (p_1)_*\pi(f) + (p_2)_*\pi(g)$. Since $(p_1)_*\pi = \mu$ and $(p_2)_*\pi = \nu$, we have the following.

\begin{lemma}\label{lem:coupling_characterization_forward} Let $\pi\in \Pi(\mu,\nu)$. Then $\pi(f\oplus g) = \mu(f) + \nu(g)$ for all $f,g\in \LipX$.
\end{lemma}

The converse of \cref{lem:coupling_characterization_forward} is true under additional hypotheses. 

\begin{proposition} \label{prop:coupling-lipc}
    Assume that $X$ is locally compact.
    Let $\mu,\nu \in \pRadon$ and $\pi \in \pBorelXX$.
    Then the following are equivalent.
    \begin{enumerate}
        \item $\pi \in \Pi(\mu,\nu)$.
        \item For all $f,g \in \LipcXpos$, $\pi(f \oplus g) = \mu(f) + \nu(g)$.
    \end{enumerate}
\end{proposition}

\begin{proof}
    The forward direction is \cref{lem:coupling_characterization_forward}.
    For the reverse direction,
    by \cref{lem:products}\eqref{it:induced-map}, $(p_1)_* \pi \in \pBorel$.
    Let $f \in \LipcXpos$.
    By \cref{lem:products}\eqref{it:pi-f-oplus-g}, $((p_1)_* \pi)(f) = \pi(f \oplus 0)$,
    which by assumption equals $\mu(f)$,
    which is finite by \cref{cor:measure_spaces_equiv_def_3}.
    Hence, by \cref{cor:measure_spaces_equiv_def_3}, $(p_1)_* \pi \in \pRadon$.
    Since for all $f \in \LipcXpos$, $((p_1)_* \pi)(f) = \mu(f)$, 
    by \cref{cor:cpt_lip_determines_pRadon}, $(p_1)_* \pi = \mu$.
    Similarly $(p_2)_* \pi = \nu$.
    Therefore $\pi \in \Pi(\mu,\nu)$.
\end{proof}

Similarly, using 
\cref{cor:representation-lip}\eqref{it:locally-compact-sigma-compact}
instead of
\cref{cor:measure_spaces_equiv_def_3}
and 
\cref{cor:lip_determines_pfRadon} 
instead of 
\cref{cor:cpt_lip_determines_pRadon} 
we have the following.

\begin{proposition} \label{prop:coupling-pfRadon}
    Assume that $X \setminus A$ is locally compact and $\sigma$-compact.
    Let $\mu,\nu \in \pfRadon$ and $\pi \in \pBorelXX$.
    Then the following are equivalent.
    \begin{enumerate}
        \item $\pi \in \Pi(\mu,\nu)$.
        \item For all $f,g \in \LipXpos$, $\pi(f \oplus g) = \mu(f) + \nu(g)$.
    \end{enumerate}
\end{proposition}

\begin{lemma} \label{lem:reflexivity-W1_iff}
    If $\mu=\nu$ then $W_1(\mu,\nu) = 0$. 
    The converse holds if $\mu,\nu \in \pfRadon$.
\end{lemma}

\begin{proof}
    Suppose that $\mu = \nu$.
    By \cref{ex:diagonal-coupling}, we have the diagonal coupling $\Delta_*\mu$.
    Then $\Delta_* \mu (\bar{d}) = \mu( \bar{d} \circ \Delta) = \mu(0) = 0$.
    Therefore $W_1(\mu,\mu) = 0$.

    Suppose that $W_1(\mu,\nu) = 0$. Let $f\in \LipX$. Given $\eps>0$, there exists $\pi\in \Pi(\mu,\nu)$ with $\pi(\bar{d}) <\eps/L(f)$. 
    By \cref{lem:lipschitz-dbar,lem:coupling_characterization_forward}, we have $\eps >L(f)\pi(\bar{d}) \geq \pi(f\oplus(-f)) = \mu(f)-\nu(f)$, and hence $\mu(f) < \nu(f) + \eps$. Similarly, $\nu(f)<\mu(f) + \eps$. Thus $\mu(f) = \nu(f)$ for all $f\in \LipX$. 
    If $\mu,\nu \in \pfRadon$, then
    by \cref{cor:lip_determines_pfRadon}, $\mu = \nu$.
\end{proof}

\begin{lemma} \label{lem:symmetry-W1}
    $W_1(\mu,\nu) = W_1(\nu,\mu)$.
\end{lemma}

\begin{proof}
    Consider the transpose map $t:X \times X \to X \times X$ given by $(x,y) \mapsto (y,x)$.
    For $\pi \in \Pi(\mu,\nu)$, $t_* \pi \in \Pi(\nu,\mu)$ and 
    $(t_* \pi) (\bar{d}) = \pi (\bar{d} \circ t) = \pi (\bar{d})$.
    The result follows.
\end{proof}

\begin{lemma}\label{lem:w1-mu-0}
    If $\mu$ is $1$-finite then $W_1(\mu,0) = \mu(d_A)$.
\end{lemma}

\begin{proof}
    Let $a \in A$.
    Consider the trivial coupling $(i_1^a)_* \mu \in \Pi(\mu,0)$.
    Then $(i_1^a)_* \mu (\bar{d}) = \mu (\bar{d} \circ i_1^a) = \mu(d_A)$.
    Hence $W_1(\mu,0) \leq \mu(d_A)$.

    Consider $\pi \in \Pi(\mu,0)$.
    Since $(p_2)_* \pi = 0$,
    $\pi(X \times (X \setminus A)) = \pi( p_2^{-1}(X \setminus A)) = (p_2)_* \pi(X \setminus A) = 0$.
    Thus, $\supp(\pi) \subset X \times A$.
    For $x \in X$, $a \in A$, $\bar{d}(x,a) = d_A(x)$. 
    Therefore $\pi(\bar{d}) \geq \pi(d_A \circ p_1) = (p_1)_* \pi(d_A) = \mu(d_A)$.
    Hence $W_1(\mu,0) \geq \mu(d_A)$.
\end{proof}

\begin{lemma} \label{lem:isometric-embedding-W1}
    Let $x,y \in X$. Then $W_1(\delta_x,\delta_y) = \bar{d}(x,y)$.
\end{lemma}

\begin{proof}
    Consider $\delta_{(x,y)} \in \Pi(\delta_x,\delta_y)$.
    Then $\delta_{(x,y)}(\bar{d}) = \bar{d}(x,y)$.
    Hence $W_1(\delta_x,\delta_y) \leq \bar{d}(x,y)$.
    It remains to prove that $W_1(\delta_x,\delta_y) \geq \bar{d}(x,y)$.
    We give both a functional analytic and a measure theoretic proof.
    
Let $B = A\cup \{x,y\}$ and define a function $f:B \to \R$ by $f(x) = d_A(x)$, $f(y) = d_A(x) - \bar{d}(x,y)$, and $f(a) = 0$ for all $a\in A$. Then $f(s)-f(t)\leq \bar{d}(s,t)$ for all $s,t\in B$. By McShane's extension theorem, $f$ extends to a function $\bar{f}:X\to \R$ which satisfies $\bar{f}(p)-\bar{f}(q)\leq \bar{d}(p,q)$ for all $p,q\in X$. Hence $\bar{f}\oplus(-\bar{f})\leq \bar{d}$. By \cref{lem:coupling_characterization_forward}, $\pi(\bar{d}) \geq \pi(\bar{f}\oplus(-\bar{f})) = \delta_x(\bar{f}) -\delta_y(\bar{f}) = f(x) - f(y) = \bar{d}(x,y)$, and hence $W_1(\delta_x,\delta_y)\geq \bar{d}(x,y)$.

    Consider $\pi \in \Pi(\delta_x,\delta_y)$.
    Then $\pi = c \delta_{(x,y)} + \rho$ for some $c \in [0,1]$ and some $\rho \in \pBorelXX$ with
    $\supp(\rho) \subset \{x\} \times A  \cup A  \times \{y\}$. 
    On $\supp(\rho)$, $\bar{d} = d_A\oplus d_A$.
    Then $\pi(\bar{d}) 
    = c\bar{d}(x,y) + \rho(d_A\oplus d_A)$. 
    Furthermore, $(p_1)_*\rho = (1-c)\delta_x$ and $(p_2)_*\rho = (1-c)\delta_y$. 
    So $\rho(d_A\oplus d_A) = (1-c)(\delta_x(d_A) + \delta_y(d_A)) = (1-c)(d_A\oplus d_A)(x,y) \geq (1-c)\bar{d}(x,y)$.
    Hence $W_1(\delta_x,\delta_y) \geq \bar{d}(x,y)$.
\end{proof}







\begin{lemma} \label{lem:finiteness}
    Let $\eps \geq 0$. 
    Let $\pi \in \Pi(\mu,\nu)$.
    If $\mu_\eps,\nu_\eps$ are finite then so is $\pi_\eps$.
\end{lemma}

\begin{proof}
    Assume that $\mu_\eps$ and $\nu_\eps$ are finite.
    Then $\pi(A_\eps \times X) = \pi (p_1)^{-1}(A_\eps) = \mu(A_\eps) = \mu_\eps(X) < \infty$.
    Similarly, $\pi(X \times A_\eps) = \nu_\eps(X) < \infty$.
    Therefore $\pi_\eps(X) = \pi((A \times A)_\eps) = \pi(A_\eps \times X \cup X \times A_\eps) \leq \pi(A_\eps \times X) + \pi(X \times A_\eps) < \infty$.
\end{proof}

\begin{corollary} \label{lem:upper-finite}
    Let $\pi \in \Pi(\mu,\nu)$.
    If $\mu$ and $\nu$ are upper finite then so is $\pi$.
\end{corollary} 


\begin{corollary} \label{cor:finite}
    Let $\pi \in \Pi(\mu,\nu)$.
    If $\mu$ and $\nu$ are finite then so is $\pi$.
\end{corollary}

\begin{lemma} \label{lem:representative-coupling}
    Let $\sigma \in \Pi(\mu,\nu)$ with $\mu,\nu$ finite.
    Let $\check{\mu} \in \pBorelAonly$ be finite with
    $\check{\mu} \geq (p_1)_*(\sigma|_{A \times (X \setminus A)})$.
    Then $\sigma$ has a representative $\sigma_1 \in \pBorelXXonly$ such that 
    $(p_1)_*\sigma_1 = \check{\mu} + \mu|_{X \setminus A}$.
    Similarly, let $\check{\nu} \in \pBorelAonly$ be finite with
    $\check{\nu} \geq (p_2)_*(\sigma|_{(X \setminus A) \times A})$.
    Then $\sigma$ has a representative $\sigma_2 \in \pBorelXXonly$ such that 
    $(p_2)_*\sigma_2 = \check{\nu} + \nu|_{X \setminus A}$.
\end{lemma}

\begin{proof}
    Let $\hat{\sigma} \in \pBorelXXonly$ denote the canonical representative of $\sigma$ with $\sigma(A^2) = 0$, and similarly for $\hat{\mu}$ and $\hat{\nu}$.
    Then $(p_1)_*\hat{\sigma} = \hat{\mu} + (p_1)_*(\sigma|_{A \times (X \setminus A)})$.
    Let $\sigma_1 = \hat{\sigma} + (i_1^a)_*(\check{\mu} - (p_1)_*(\sigma|_{A \times (X \setminus A)}))$.
    Then $\sigma_1$ is the desired representative of $\sigma$.
    The other case is similar.
\end{proof}

Assume $\eps > 0$.
Let $a \in A$.
We will define a (discontinuous) retraction from $X^2$ to $(A \cup A_\eps)^2$.
Let $\hat{r}:X \to X$ be given by
\begin{equation} \label{eq:retraction}
    \hat{r}(x) = \begin{cases}
        a &\text{if } x \in A_0^\eps\\
        x &\text{otherwise,}
    \end{cases}
\end{equation}
Then we have the desired retract $r:X^2 \to X^2$ given by $r = \hat{r} \oplus \hat{r}$, i.e., $r(x,y) = (\hat{r}(x),\hat{r}(y))$.



\begin{lemma} \label{lem:r-pi-coupling}
    Let $\pi \in \Pi(\mu,\nu)$. Then
    $r_*\pi \in \Pi(\mu_\eps,\nu_\eps)$
\end{lemma}

\begin{proof}
    First note that $\hat{r}_*\mu = \mu_\eps$, since
    $(\mu \circ \hat{r})|_{A_\eps} = \mu|_{A_\eps}$ and $(\mu \circ \hat{r})|_{A^\eps} = 0$. 
    Next note that $p_1\circ r = \hat{r}\circ p_1$, since both send $(x,y)$ to $\hat{r}(x)$. 
    Then $(p_1)_*r_*\pi = (p_1\circ r)_*\pi = (\hat{r}\circ p_1)_*\pi = \hat{r}_*(p_1)_*\pi = \hat{r}_*\mu = \mu_\eps$. 
    Similarly, $(p_2)_* r_* \pi = \nu_\eps$.
    Hence $r_*\pi \in \Pi(\mu_\eps,\nu_\eps)$.
\end{proof}



\begin{proposition} \label{prop:r-pi-cost}
    Let $\pi \in \Pi(\mu,\nu)$. Then
    $r_*(\pi)(\bar{d}) \leq \pi(\bar{d}) + \mu^\eps(d_A) + \nu^\eps(d_A)$.
\end{proposition}

\begin{proof}
Note that $r_*\pi(\bar{d}) + r_*\pi(d_A\oplus d_A) \leq \pi(\bar{d}) + \pi(d_A\oplus d_A)$ so that $r_*\pi(\bar{d})  \leq \pi(\bar{d}) + \pi(d_A\oplus d_A) -  r_*\pi(d_A\oplus d_A)$. 
By \cref{lem:products}\eqref{it:pi-f-oplus-g},
$\pi(d_A \oplus d_A) = \mu(d_A) + \nu(d_A)$ and
$r_*\pi(d_A \oplus d_A) = \pi((d_A \circ \hat{r}) \oplus (d_A \circ \hat{r})) = \mu(d_A \circ \hat{r}) + \nu(d_A \circ \hat{r}) = \mu_\eps(d_A) + \nu_\eps(d_A)$.
Thus,
$\pi(d_A\oplus d_A) - r_*\pi(d_A \oplus d_A) 
= \mu^\eps(d_A) + \nu^\eps(d_A)$, and the result follows.
\end{proof}

Note that $r_*\pi$ has a trivial extension to a coupling of $\mu$ and $\nu$ given by
$r_*(\pi) + (i_1^c)_* (\mu^\eps) + (i_2^c)_* (\nu^\eps)$.

Recall that $\fpRadon 
= \{\mu \in \pBorel \ | \ \mu \text{ is tight, and } \mu(X \setminus A) < \infty\}$.
Let $p_{12},p_{23}:X^3 \to X^2$ denote the projections on the first and second, and first and third coordinates, respectively.

\begin{proposition} \label{thm:triangle-inequality-W1}
    $W_1$ satisfies the triangle inequality on $\pfRadon$.
\end{proposition}

\begin{proof}
    Let $\mu_1,\mu_2,\mu_3 \in \pfRadon$. 
    Let $\eps > 0$. 
    Let $\pi_{12} \in \Pi(\mu_1,\mu_2)$ and
    $\pi_{23} \in \Pi(\mu_2,\mu_3)$
    such that $\pi_{12}(\bar{d}) < W_1(\mu_1,\mu_2) + \frac{\eps}{8}$ and
    $\pi_{23}(\bar{d}) < W_1(\mu_2,\mu_3) + \frac{\eps}{8}$.
    By \cref{lem:lower-p-finite}, 
    there is a $\delta>0$ such that for $j=1,2,3$, $(\mu_j)^\delta(d_A) < \frac{\eps}{8}$.
    For $\delta$, use \eqref{eq:retraction} to define a retract $r$ from $X^2$ to $(A \cup A_\delta)^2$.
    Consider $r_*(\pi_{12})$ and $r_*(\pi_{23})$.
    By \cref{lem:r-pi-coupling},
    $r_*(\pi_{12}) \in \Pi((\mu_1)_\delta,(\mu_2)_\delta)$ and
    $r_*(\pi_{23}) \in \Pi((\mu_2)_\delta,(\mu_3)_\delta)$.
    By \cref{lem:mu-eps-finite}, $(\mu_1)_\delta,(\mu_2)_\delta,(\mu_3)_\delta \in \fpRadon$.
    By \cref{prop:r-pi-cost}, 
    $r_*(\pi_{12})(\bar{d}) \leq \pi_{12}(\bar{d}) + (\mu_1)^\delta(d_A) + (\mu_2)^\delta(d_A) < \pi_{12}(\bar{d}) + \frac{\eps}{4}$ and
    $r_*(\pi_{23})(\bar{d}) \leq \pi_{23}(\bar{d}) + (\mu_2)^\delta(d_A) + (\mu_3)^\delta(d_A) < \pi_{23}(\bar{d}) + \frac{\eps}{4}$.

    Let $\check{\mu}_2 = (p_2)_*(r_*(\pi_{12})|_{(X \setminus A) \times A}) \join (p_1)_*(r_*(\pi_{23})|_{A \times (X \setminus A)})$.
    By \cref{lem:representative-coupling},
    $r_*(\pi_{12})$ has a finite representative $\sigma_2 \in \pBorelXXonly$
    such that 
    $(p_2)_*\sigma_2 = \check{\mu}_2 + (\mu_2)_\delta$ and
    $r_*(\pi_{23})$ has a finite representative $\sigma_1 \in \pBorelXXonly$
    such that 
    $(p_1)_*\sigma_1 = \check{\mu}_2 + (\mu_2)_\delta$.
    Let $m = (\check{\mu}_2 + (\mu_2)_\delta)(X)$.
    Let $\gamma_{12} = \frac{1}{m} \sigma_2$ and
    let $\gamma_{23} = \frac{1}{m} \sigma_1$.
    By the gluing lemma for probability measures on a Polish space~\cite[Lemma 5.3.2]{Ambrosio:2008}, there exists a probability measure 
    $\gamma$ on $X^3$ such that 
    $(p_{12})_*(\gamma) = \gamma_{12}$ and
    $(p_{23})_*(\gamma) = \gamma_{23}$,
    where $p_{12},p_{23}:X^3 \to X^2$ denote the projections on the first two and last two coordinates, respectively.
    Let $\gamma_{13} = (p_{13})_*\gamma$, where $p_{13}:X^3 \to X^2$ denotes the projection onto the first and third coordinates.
    Then $m \gamma_{13}$ represents a coupling $\pi'_{13}$ of $(\mu_1)_\delta$ and $(\mu_3)_\delta$.
    Let $\pi = m \gamma$.

%
    By the triangle inequality,
    \begin{multline*}
        \pi_{13}'(\bar{d}) = 
        \int_{X^2} \bar{d}(x,z) d\pi'_{13}(x,z) = \int_{X^3} \bar{d}(x,z) d\pi(x,y,z)
        \leq \int_{X^3} \left( \bar{d}(x,y) + \bar{d}(y,z) \right) d\pi(x,y,z)\\
        = \int_{X^2} \bar{d}(x,y) d\sigma_2(x,y) + 
        \int_{X^2} \bar{d}(y,z) d\sigma_1(y,z)
        = r_*(\pi_{12})(\bar{d}) + r_*(\pi_{23})(\bar{d}).
    \end{multline*}
    Fix $a\in A$ and extend $\pi_{13}'$ trivially to a coupling $\pi_{13}$ of $\mu_1$ and $\mu_3$.  
    That is, $\pi_{13} = \pi_{13}' + (i_1^a)_*((\mu_1)^\delta) + (i_2^a)_*((\mu_2)^\delta)$ and $\pi_{13} \in \Pi(\mu_1,\mu_3)$.
    Then $\pi_{13}(\bar{d}) = \pi_{13}'(\bar{d}) + (\mu_1)^\delta(d_A) + (\mu_2)^\delta(d_A) < \pi_{13}'(\bar{d}) + \frac{\eps}{4}$.
    Hence $\pi_{13}(\bar{d}) < r_*(\pi_{12})(\bar{d}) + r_*(\pi_{23})(\bar{d}) + \frac{\eps}{4} < \pi_{12}(\bar{d}) + \pi_{23}(\bar{d}) + \frac{3\eps}{4} < W_1(\mu_1,\mu_2) + W_1(\mu_2,\mu_3) + \eps$.
    Therefore $W_1(\mu,\mu_3) \leq W_1(\mu_1,\mu_2) + W_1(\mu_2,\mu_3)$.
\end{proof}

Combining \cref{lem:reflexivity-W1_iff,lem:symmetry-W1,thm:triangle-inequality-W1,lem:isometric-embedding-W1}, we have the following.

\begin{theorem} \label{thm:metric}
    $(\pfRadon, W_1)$ is a metric space, and the inclusion $X \to \pfRadon$ given by $x \mapsto \delta_x$ gives an isometric embedding $(X,\bar{d}) \to (\pfRadon,W_1)$.
\end{theorem}

\begin{proposition}
    Let $\mu \in \pfRadon$ then there exists a sequence $(\mu^{(n)}) \subset \fpRadon \cap \pfRadon$ such that $\mu^{(n)} \to \mu$ in $(\pfRadon,W_1)$.
\end{proposition}

\begin{proof}
    By \cref{lem:lower-p-finite},
    for all $n \geq 1$,
    there exists $\delta > 0$ such that $\mu^\delta(d_A) < \frac{1}{n}$.
    By \cref{ex:approximation-by-finite-measure}, 
    $W_1(\mu,\mu_\delta) \leq \mu^\delta(d_A) < \frac{1}{n}$.
    Since $\mu_\delta \leq \mu$, $\mu_\delta \in \pfRadon$.
    By \cref{lem:mu-eps-finite}, $\mu_\delta \in \fpRadon$.
    Let $\mu^{(n)} = \mu_\delta$.
\end{proof}

\subsection{p-Wasserstein distance for metric pairs}
\label{sec:p-wasserstein}

Let $(X,d,A)$ be a metric pair with 
$(X,d)$ complete and separable and
$A \neq \emptyset$.
Recall \eqref{eq:d-bar}, $\bar{d} = d \wedge d_A \oplus d_A = d \wedge (d+d)_{A \times A}$.
That is, $\bar{d}(x,y) = d(x,y) \wedge (d_A(x) + d_A(y))$.
For $1 \leq p < \infty$, define
\begin{equation}
    d_p = d \wedge (d_A^p \oplus d_A^p)^\frac{1}{p}.
\end{equation}
That is, $d_p(x,y) = d(x,y) \wedge \norm{(d_A(x),d_A(y))}_p$.
In particular, $d_1 = \bar{d}$.
One can check that $d_p$ is pseudometric on $(X,d,A)$~\cite[Lemma 3.13]{bubenik2022universality} and $d_p(x,y) = 0$ if and only if either $x=y$ or $x \in A$ and $y \in A$,
and $d_p$ is a metric on the quotient $X/A$ \cite[Lemma 3.17]{bubenik2022universality}.

Let $\mu,\nu \in \pBorel$.

\begin{definition} \label{def:Wp-hat}
    Let $1 \leq p < \infty$.
    Define 
    \begin{equation*}
        \hat{W}_p(\mu,\nu)  
        = \left( \inf_{\pi \in \Pi(\mu,\nu)} \pi( d_p^p ) \right)^{\frac{1}{p}}.
    \end{equation*}
\end{definition}

\begin{definition} \label{def:Wp}
    Let $a \in A$.
    For $\eps > 0$, let 
    \[
    \Pi_\eps(\mu,\nu) = \{ \pi \in \Pi(\mu,\nu) \ | \ \exists \pi' \in \Pi(\mu_\eps,\nu_\eps) \text{ such that } 
    \pi = \pi' + (i_1^a)_* \mu^\eps + (i_2^a)_* \nu^\eps\}.
    \]
    Note that $\Pi_{\eps}(\mu,\nu)\neq\emptyset$ since every coupling $\pi'\in \Pi(\mu_\eps,\nu_\eps)$ can be trivially extended to a coupling $\pi\in \Pi(\mu,\nu)$ by \cref{ex:trivial-coupling}.
    Define
    the \emph{$p$-Wasserstein distance} between $\mu$ and $\nu$ to be given by 
    \begin{equation*}
        W_p(\mu,\nu) = \left( \inf_{\eps > 0} \inf_{\pi \in \Pi_\eps(\mu,\nu)} \pi(d_p^p) \right)^{\frac{1}{p}}.
    \end{equation*}
\end{definition}

Note that if $\pi = \pi' + (i_1^c)_* \mu^\eps + (i_2^c)_* \nu^\eps$ 
then $\pi(d_p^p) = \pi'(d_p^p) + \mu^\eps(d_A^p) + \nu^\eps(d_A^p)$.

\begin{lemma}
    If $0 < \eps < \eps'$ then $\Pi_{\eps'}(\mu,\nu) \subset \Pi_\eps(\mu,\nu)$.
\end{lemma}

\begin{proof}
    By \cref{ex:trivial-coupling}, we can trivially extend a coupling of $\mu_{\eps'}$ and $\nu_{\eps'}$ to a coupling of $\mu_\eps$ and $\nu_\eps$.
\end{proof}

From the definitions we have the following.

\begin{lemma} \label{lem:Wp-inequality}
    $\tilde{W}_p(\mu,\nu) \leq W_p(\mu,\nu)$.
\end{lemma}


\begin{proposition}
    $\hat{W}_1(\mu,\nu) = W_1(\mu,\nu)$.
\end{proposition}

\begin{proof}
    By \cref{lem:Wp-inequality}, it remains to show that $W_1(\mu,\nu) \leq \hat{W}_1(\mu,\nu)$.
    Let $\eps > 0$.
    By \cref{lem:lower-p-finite}, there exists $\delta >0$ such that $\mu^\delta(d_A) < \frac{\eps}{3}$ and $\nu^\delta(d_A) < \frac{\eps}{3}$.
    By \cref{def:Wp-hat}, there exists $\pi \in \Pi(\mu,\nu)$ such that $\pi(d_1) < \hat{W}_1(\mu,\nu) + \frac{\eps}{3}$.
    Use $\delta$ and \eqref{eq:retraction} to define the retraction $r$ and let $\pi' = r_*(\pi)$.
    By \cref{lem:r-pi-coupling}, $\pi' \in \Pi(\mu_\delta,\nu_\delta)$.
    By \cref{prop:r-pi-cost}, $\pi'(d_1) \leq \pi(d_1) + \mu^\delta(d_A) + \nu^\delta(d_A) < \hat{W}_1(\mu,\nu) + \eps$.
    Therefore, by \cref{def:Wp}, $W_1(\mu,\nu) < \hat{W}_1(\mu,\nu) + \eps$ and hence $W_1(\mu,\nu) \leq \hat{W}_1(\mu,\nu)$.
\end{proof}

It is an open question whether $\hat{W}_p = W_p$ for $p > 1$.

\begin{lemma} \label{lem:reflexivity-Wp}
    $W_p(\mu,\mu) = 0$.
\end{lemma}

\begin{proof}
    Let $\eps>0$.
    by \cref{lem:lower-p-finite}, there exists $\delta > 0$ such that $\mu^\delta(d_A^p) < \frac{\eps}{2}$.
    By \cref{ex:diagonal-coupling}, we have the diagonal coupling $\pi' = \Delta_*\mu_\delta$.
    By \cref{ex:trivial-extension}, extend $\pi'$ trivially to a coupling $\pi$ of $\mu$ and $\mu$.
    Then $\pi(d_p^p) = \pi'(d_p^p) + 2\mu^\delta(d_A^p) < \eps$.
    Therefore $W_p(\mu,\mu) = 0$.
\end{proof}

\begin{lemma} \label{lem:symmetry-Wp}
    $W_p(\mu,\nu) = W_p(\nu,\mu)$.
\end{lemma}

\begin{proof}
    Let $\eps > 0$ and let $\pi \in \Pi_\eps(\mu,\nu)$.
    Then using the transpose map $t$, 
    $t_*\pi \in \Pi_{\eps}(\nu,\mu)$ and $(t_*\pi)(d_p^p) = \pi(d_p^p)$.
    The result follows.
\end{proof}

\begin{proposition} \label{thm:triangle-inequality-Wp}
    $W_p$ satisfies the triangle inequality on $\pfpRadon$.
\end{proposition}

\begin{proof}
    For $j=1,2,3$, let $\mu_j \in \pfpRadon$.
    Let $\eps > 0$.
    By \cref{lem:lower-p-finite,def:Wp}, there is a $\delta > 0$ such that 
    for $j=1,2,3$, $(\mu_j)^\delta(d_A^p) < \frac{\eps}{2}$ and
    there exists $\pi_{12} \in \Pi((\mu_1)_\delta,(\mu_2)_\delta)$ such that 
    $\pi_{12}(d_p^p)^{\frac{1}{p}} < W_p(\mu_1,\mu_2) + \frac{\eps}{2}$
    and
    there exists $\pi_{23} \in \Pi((\mu_2)_\delta,(\mu_3)_\delta)$ such that 
    $\pi_{23}(d_p^p)^{\frac{1}{p}} < W_p(\mu_2,\mu_3) + \frac{\eps}{2}$.
    For $j=1,2,3$, $\mu_j$ is upper $p$-finite and hence upper-finite and
    thus $(\mu_j)_\delta$ is finite.

    Let $\check{\mu}_2 = 
    (p_2)_*(\pi_{12}|_{(X \setminus A) \times A}) \join 
    (p_1)_*(\pi_{23}|_{A \times (X \setminus A)})$.
    By \cref{lem:representative-coupling},
    $\pi_{12}$ has a finite representative $\sigma_2 \in \pBorelXXonly$ such that $(p_2)_*\sigma_2 = \check{\mu}_2 + (\mu_2)_\delta$, and
    $\pi_{23}$ has a finite representative $\sigma_1 \in \pBorelXXonly$ such that $(p_1)_*\sigma_1 = \check{\mu}_2 + (\mu_2)_\delta$.
    Let $m = (\check{\mu}_2 + (\mu_2)_\delta)(X)$.
    Let $\gamma_{12} = \frac{1}{m} \sigma_2$ and
    let $\gamma_{23} = \frac{1}{m} \sigma_1$.
    By the gluing lemma for probability measures on a Polish space~\cite[Lemma 5.3.2]{Ambrosio:2008}, there exists a probability measure 
    $\gamma$ on $X^3$ such that 
    $(p_{12})_*(\gamma) = \gamma_{12}$ and
    $(p_{23})_*(\gamma) = \gamma_{23}$,
    where $p_{12},p_{23}:X^3 \to X^2$ denote the projections on the first two and last two coordinates, respectively.
    Let $\gamma_{13} = (p_{13})_*\gamma$, where $p_{13}:X^3 \to X^2$ denotes the projection onto the first and third coordinates.
    Then $m \gamma_{13}$ represents a coupling $\pi_{13}$ of $(\mu_1)_\delta$ and $(\mu_3)_\delta$.
    Let $\pi = m \gamma$.
    By the triangle inequality and the Minkowski inequality,
    \begin{multline*}
        \pi_{13}(d_p^p)^\frac{1}{p} 
        = \left( \int_{X^2} d_p(x,z)^p d\pi_{13}(x,z) \right)^\frac{1}{p} 
        = \left( \int_{X^3} d_p(x,z)^p d\pi(x,y,z) \right)^\frac{1}{p}
        = \norm{d_p\circ p_{13}}_{p,\pi}
        \\ 
        \leq \norm{d_p\circ p_{12} + d_p\circ p_{23}}_{p,\pi} = \norm{d_p}_{p,\pi_{12}} + \norm{d_p}_{p,\pi_{23}}
        = \pi_{12}(d_p^p)^\frac{1}{p} + \pi_{23}(d_p^p)^\frac{1}{p},
    \end{multline*}
    where $\norm{f}_{p,\sigma} = (\int f^p d\sigma)^\frac{1}{p}$.
    Therefore,
    \begin{multline*}
        W_p(\mu_1,\mu_3) 
        \leq \left( (\mu_1)^\delta(d_A^p) + (\mu_3)^\delta(d_A^p) + \pi_{13}(d_p^p) \right)^\frac{1}{p}
        < \left( \eps + \left( \pi_{12}(d_p^p)^\frac{1}{p} + \pi_{23}(d_p^p)^\frac{1}{p} \right)^p \right)^\frac{1}{p}\\
        < \left( \eps + ( W_p(\mu_1,\mu_2) + W_p(\mu_2,\mu_3) + \eps)^p \right)^\frac{1}{p}.
    \end{multline*}
    Since $\eps>0$ was arbitrary, $W_p(\mu_1,\mu_3) \leq W_p(\mu_1,\mu_2) + W_p(\mu_2,\mu_3)$.
\end{proof}

\begin{lemma} \label{lem:Wp-mu-0}
    If $\mu \in \pfpRadon$ then $W_p(\mu,0) = \mu(d_A^p)^{\frac{1}{p}}$.
\end{lemma}

\begin{proof}
    Let $a \in A$.
    Consider the trivial coupling $\pi = (i_1^a)_* \mu \in \Pi(\mu,0)$.
    Then for all $\eps > 0$, $\pi \in \Pi_\eps(\mu,0)$ and $\pi(d_p^p) = \mu(d_p^p \circ i_1^a) = \mu(d_A^p)$.
    Hence $W_p(\mu,0) \leq \mu(d_A^p)^\frac{1}{p}$.

    Consider $\pi \in \Pi(\mu,0)$.
    Since $(p_2)_* \pi = 0$,
    $\pi(X \times (X \setminus A)) = \pi( p_2^{-1}(X \setminus A)) = ((p_2)_* \pi)(X \setminus A) = 0$.
    Thus, $\supp(\pi) \subset X \times A$.
    For $x \in X$, $a \in A$, $d_p(x,a) \geq d_A(x)$. 
    Therefore $\pi(d_p^p) \geq \pi(d_A^p \circ p_1) = (p_1)_* \pi(d_A^p) = \mu(d_A^p)$.
    Hence $W_p(\mu,0) \geq \mu(d_A^p)^\frac{1}{p}$.
\end{proof}

\begin{lemma} \label{lem:isometric-embedding-Wp}
    Let $x,y \in X$. Then $W_p(\delta_x,\delta_y) = d_p(x,y)$.
\end{lemma}

\begin{proof}
    If $y \in A$ then by \cref{lem:Wp-mu-0}, $W_p(\delta_x,\delta_y) = W_p(\delta_x,0) = \delta_x(d_A^p)^\frac{1}{p} = d_A(x) = d_p(x,y)$.
    Similarly if $x \in A$ then $W_p(\delta_x,\delta_y) = d_p(x,y)$.
    Assume $x,y \in X \setminus A$.

    Consider $\delta_{(x,y)} \in \Pi(\delta_x,\delta_y)$.
    For all $0 < \eps < d_A(x) \meet d_A(y)$, $\delta_{(x,y)} \in \Pi_\eps(\delta_x,\delta_y)$.
    Since $\delta_{(x,y)}(d_p^p) = d_p(x,y)^p$, 
    $W_p(\delta_x,\delta_y) \leq d_p(x,y)$.

    Next, consider $\pi \in \Pi(\delta_x,\delta_y)$.
    Then $\pi = c \delta_{(x,y)} + \rho$ for some $c \in [0,1]$ and some $\rho \in \pBorelXX$ with
    $\supp(\rho) \subset \{x\} \times A \cup A \times \{y\}$ such that 
    $\rho(\{x\} \times A) = 1-c$ and $\rho(A \times \{y\}) = 1-c$.
    For all $a \in A$, $d_p(x,a) \geq d_A(x)$. 
    Similarly, for all $a \in A$, $d_p(a,y) \geq d_A(y)$. 
    Therefore,
    $\pi(d_p^p) 
    \geq c d_p(x,y)^p + (1-c)(d_A^p \circ p_1)(x,y) + (1-c)(d_A^p \circ p_2)(x,y)
    = c d_p(x,y)^p + (1-c) (d_A^p \oplus d_A^p)(x,y) \geq d_p(x,y)^p$.
    Hence $W_p(\delta_x,\delta_y) \geq d_p(x,y)$.
\end{proof}

Combining \cref{lem:reflexivity-Wp,lem:symmetry-Wp,thm:triangle-inequality-Wp,lem:isometric-embedding-Wp}, we have the following.

\begin{theorem}
    $(\pfpRadon, W_p)$ is a pseudometric space, and the inclusion $X \to \pfpRadon$ given by $x \mapsto \delta_x$ gives an isometric embedding $(X,d_p) \to (\pfpRadon,W_p)$.
\end{theorem}

It is an open question whether $W_p(\mu,\nu)=0$ implies that $\mu = \nu$ when $p > 1$.

\begin{proposition}
    Let $\mu \in \pfpRadon$ then there exists a sequence $(\mu^{(n)}) \subset \fpRadon \cap \pfpRadon$ such that $\mu^{(n)} \to \mu$ in $(\pfpRadon,W_p)$.
\end{proposition}

\begin{proof}
    By \cref{lem:lower-p-finite},
    for all $n \geq 1$,
    there exists $\delta > 0$ such that $\mu^\delta(d_A^p) < \frac{1}{n}$.
    By \cref{ex:approximation-by-finite-measure}, 
    $W_p(\mu,\mu_\delta) \leq \mu^\delta(d_A) < \frac{1}{n}$.
    Since $\mu_\delta \leq \mu$, $\mu_\delta \in \pfpRadon$.
    By \cref{lem:mu-eps-finite}, $\mu_\delta \in \fpRadon$.
    Let $\mu^{(n)} = \mu_\delta$.
\end{proof}

\section{Duality}
\label{sec:KR}

In this section, we prove the analogs of Monge-Kantorovich duality and Kantorovich-Rubinstein duality in the relative setting. Our Kantorovich-Rubinstein duality allows for a nice description of the operator norm of order bounded, sequentially order continuous linear functionals on $\LipX$ and $\LipcX$ in terms of the relative Wasserstein distance, leading to a strengthening of Theorems \ref{thm:representation-proper-not-positive} and \ref{thm:representation-lip}.
In this section assume that all metric spaces are
complete and separable.

\subsection{Monge-Kantorovich duality and Kantorovich-Rubinstein duality}
\label{sec:mk-duality-kr-duality}

In this section, we prove analogs of the classical Monge-Kantorovich duality and Kantorovich-Rubinstein duality in the relative setting.
%
  Our main tools are the Hahn-Banach theorem, Theorem~\ref{thm:positive_hahn_banach_strong}, 
  and the representation theorems, Theorems \ref{thm:representation-lipc} and \ref{thm:representation-lip}.
Many of the main ideas here are due to Edwards, who used them to give a proof of the classical Kantorovich-Rubinstein duality theorem \cite{edwards2010simple}.
In the other direction,
we view Kantorovich-Rubinstein duality as a strengthening of Theorem \ref{thm:representation-lip}. We observe that the operator norm of an order bounded, sequentially order continuous linear functional on $\LipX$ can be described in terms of the the relative $1$-Wasserstein distance, from which we obtain our main results, Theorem \ref{thm:main_result} and Theorem~\ref{thm:main_result_compact}.

Let $(X,A)$ denote the metric par $(X,d,A)$ and let $(X^2,A^2)$ denote $(X \times X, d + d, A \times A)$.
Recall that the projections $p_1,p_2:X\times X\to X$ 
induce $1$-Lipschitz morphisms $p_1,p_2:(X^2,A^2) \to (X,A)$.
Also recall that $(d+d)_{A \times A} = d_A \oplus d_A$ is an order unit for $\LipXXAA$ and for $h \in \LipXXAA$, $h \leq L(h)(d_A \oplus d_A)$.



\begin{definition}
  Let $\mu,\nu \in \pfRadon$.
  Define $\omega_{\mu,\nu}:\LipXXAA \to \R$ by
  \begin{equation} \label{eq:omega}
      \omega_{\mu,\nu}(h) = \inf\{\mu(f) + \nu(g) \ | f,g\in \LipX, h \leq f \oplus g\}.
  \end{equation} 
\end{definition}
%
%
Since $d_A \oplus d_A$ is an order unit for $\LipXXAA$,
the set on the right hand side of \eqref{eq:omega} is nonempty.

\begin{proposition}\label{prop:properties_of_p}
  Let $\mu,\nu \in \pfRadon$.
  Then
\begin{enumerate}
\item \label{it:omega-monotonic}
  $\omega_{\mu,\nu}$ is a monotonic sublinear functional,
\item \label{it:omega-bounds-pi}
  for all $\pi \in \Pi(\mu,\nu)$ and $h\in \LipXXAA$,
  $\pi(h) \leq \omega_{\mu,\nu}(h)$, and
\item \label{it:omega-of-f-oplus-g}
  for all $f,g\in \LipX$,
  $\omega_{\mu,\nu}(f\oplus g) = \mu(f) + \nu(g)$.
\end{enumerate}

\begin{proof}
  
  \eqref{it:omega-monotonic}
  Let $h,h'\in \LipXXAA$.
  Let $f,g \in \LipX$ such that $f \oplus g \geq h$ and
  let $f',g' \in \LipX$ such that $f' \oplus g' \geq h'$. Then $h + h' \leq (f+f')\oplus (g+g')$ and hence $\omega_{\mu,\nu}(h+h') \leq \mu(f) + \nu(g) + \mu(f') + \nu(g')$. Thus $\omega_{\mu,\nu}(h+h')\leq \omega_{\mu,\nu}(h) + \omega_{\mu,\nu}(h')$.
   For $\alpha > 0$ we have $h
   \leq f\oplus g \iff \alpha h \leq \alpha f \oplus \alpha g$ and hence $\omega_{\mu,\nu}(\alpha h) = \alpha \omega_{\mu,\nu}(h)$.
		
To see that $\omega_{\mu,\nu}$ is monotonic, let $h \leq h' \in \LipXXAA$.
Then for $f,g \in \LipX$, $f \oplus g \geq h$ whenever $f \oplus g \geq h'$. 
Hence $\omega_{\mu,\nu}(h)\leq \omega_{\mu,\nu}(h')$.
		
\eqref{it:omega-bounds-pi} Consider $\pi \in \Pi(\mu,\nu)$ and $h\in \LipXXAA$.
               Let $f,g \in \LipX$ such that $h\leq f\oplus g$. Then by \cref{lem:coupling_characterization_forward}, 
               $\pi(h) \leq \pi(f\oplus g) = \mu(f) + \nu(g)$ and hence $\pi(h) \leq \omega_{\mu,\nu}(h)$.
		
\eqref{it:omega-of-f-oplus-g} 
Let $f,g \in \LipX$.
By \cref{lem:coupling_characterization_forward}, \eqref{it:omega-bounds-pi} and \eqref{eq:omega}, for all $\pi \in \Pi(\mu,\nu)$, we have
$\mu(f) + \nu(g) = \pi(f \oplus g) \leq \omega_{\mu,\nu}(f\oplus g) \leq \mu(f) + \nu(g)$, giving the result.
	\end{proof}
\end{proposition}

\begin{theorem}\label{thm:riesz_coupling}
    Assume that 
    $(X,A)$ is boundedly compact.
  Let $\mu,\nu \in \pfRadon$.
	Let $T:\LipXXAA \to \R$ be a positive linear functional such that $T(h) \leq \omega_{\mu,\nu}(h)$ for all $h\in \LipXXAA$. Then there exists a 
        unique coupling
        $\sigma \in \Pi(\mu,\nu)$ such that $T(h) = \sigma(h)$ for all $h\in \LipXXAA$.
\end{theorem}

\begin{proof}
    Since $T$ is a positive linear functional on $\LipX$, it is order bounded, and hence
    by \cref{lem:order_bounded_implies_norm_bounded},
    $T$ is a bounded, positive linear functional on $\LipcX$.
    
    Next we show that $T$ is 
    sequentially order continuous. Let $(h_n) \subset \LipcXXAA$ such that $h_n \downarrow 0$. 
    We want to show that $T(h_n) \to 0$.
    Since $h_n$ is compactly supported, there exists a compact subset $K \subset X$ such that $\supp(h_n) \subset K \times K$ for all $n$.
    Hence $h_n \leq L(h_1)d_{(K^2)^c \cup A^2}$.
    Let $a_n = \sup(h_n)$.
    By Dini's theorem, $a_n \downarrow 0$.

    We claim that $d_{(K^2)^c \cup A^2} \leq d_{K^c \cup A} \oplus d_{K^c \cup A}$.
    Since $(K^2)^c \cup A^2 = (K^c \times X) \cup (X \times K^c) \cup A^2$,
    we have that 
    $d_{(K^2)^c \cup A^2}(x,y) = \min(d_{K^c}(x), d_{K^c}(y), d_A(x)+ d_A(y))$.
    For the right hand side,
    $(d_{K^c \cup A} \oplus d_{K^c \cup A})(x,y) = \min(d_{K^c}(x),d_A(x)) + \min(d_{K^c}(y), d_A(y)) = \min(d_{K^c}(x) + d_{K^c}(y), d_{K^c}(x) + d_A(y), d_A(x) + d_{K^c}(y), d_A(x) + d_A(y))$.
    The left hand side is less than or equal to each of the four terms above, and the claim follows.

    Therefore, $h_n \leq h_1 \leq L(h_1)d_{K^c \cup A} \oplus L(h_1) d_{K^c \cup A}$.
    Also, $h_n \leq a_n$.
    Hence, 
    $h_n \leq (L(h_1) d_{K^c \cup A} \meet a_n) \oplus (L(h_1) d_{K^c \cup A} \meet a_n)$.

    By \cref{prop:properties_of_p}\eqref{it:omega-of-f-oplus-g},
    $T(h_n) \leq \omega_{\mu,\nu}(h_n) \leq L(h_1)(\mu + \nu)(d_{K^c \cup A} \meet a_n) \downarrow 0$, since $\mu,\nu$ are sequentially order continuous by \cref{lem:sequentially-order-continuous}.   

    By \cref{thm:representation-proper}, $T$ is represented by a unique $\sigma \in \pfRadonXXAA$.
    It remains to show that $\sigma\in \Pi(\mu,\nu)$. For $f,g\in \LipX$, we have $T(f\oplus g) \leq \omega_{\mu,\nu}(f\oplus g) = \mu(f) + \nu(g)$ and $-T(f\oplus g) = T((-f)\oplus (-g)) \leq \omega_{\mu,\nu}((-f)\oplus (-g)) = -\mu(f) - \nu(g)$. Thus $\sigma(f\oplus g) = T(f\oplus g) = \mu(f) + \nu(g)$ for all $f,g\in \LipX$ so that $\sigma \in \Pi(\mu,\nu)$ by Proposition \ref{prop:coupling-pfRadon}, as desired.
\end{proof}

\begin{theorem}[Relative Monge-Kantorovich Duality]\label{thm:MK-duality}
    Assume that $(X,A)$ is boundedly compact.
    Let $\mu,\nu \in \pfRadon$ and $h\in \LipXXAApos$. Then
    \[ \min_{\pi\in \Pi(\mu,\nu)}\pi(h) = \sup\{ \mu(f) + \nu(g) \ | \ f,g \in \LipX, f \oplus g \leq h\}.\]
\end{theorem}

\begin{proof}
First, let $k = -h$. 
Then $k \leq 0$, and since $\omega_{\mu,\nu}$ is monotonic, $\omega_{\mu,\nu}(k) \leq 0$.
Let $G\subset \LipXXAA$ be the 
linear subspace spanned by $k$. 
Define $T':G\to \R$ by $T'(\alpha k) = \alpha \omega_{\mu,\nu}(k)$ for all $\alpha\in \R$. 
Then $T'$ is a linear functional on $G$.
Furthermore, $T'$ is positive, since
if $\alpha k \geq 0$ then $\alpha \leq 0$, and hence $T'(\alpha k) = \alpha \omega_{\mu,\nu}(k) \geq 0$. 
Moreover,
since $\omega_{\mu,\nu}$ is sublinear, we have $\alpha\omega_{\mu,\nu}(k)\leq \omega_{\mu,\nu}(\alpha k)$ for all $\alpha \in \R$,
and hence $T'(\alpha k) \leq \omega_{\mu,\nu}(\alpha k)$ for all $\alpha \in \R$. Thus $T'\leq \omega_{\mu,\nu}$ on $G$.

Since $k\leq 0$, $G$ equals the Riesz subspace generated by $k$.
By the Hahn-Banach theorem (\cref{thm:positive_hahn_banach_strong}), $T'$ extends to a positive linear functional $T:\LipXXAA\to \R$ such that $T(\phi) \leq \omega_{\mu,\nu}(\phi)$ for all $\phi\in \LipXXAA$.
In particular, $T(k) = \omega_{\mu,\nu}(k)$.
By Theorem \ref{thm:riesz_coupling}, there exists a unique $\sigma \in \Pi(\mu,\nu)$ such that $T(\phi) = \sigma(\phi)$ for all $\phi\in \LipXXAA$. Hence, $\sigma(k) = T(k) = \omega_{\mu,\nu}(k)$. On the other hand, $\pi(k) \leq \omega_{\mu,\nu}(k)$ for all $\pi\in \Pi(\mu,\nu)$ by Proposition \ref{prop:properties_of_p}\eqref{it:omega-bounds-pi}. Thus $\sup_{\pi\in \Pi(\mu,\nu)} \pi(k)$ is attained by $\sigma(k)$
  and so we have
\[ \max_{\pi\in \Pi(\mu,\nu)} \pi(k) = \sigma(k) = \omega_{\mu,\nu}(k) = \inf\left\{ \mu(f) + \nu(g) \ | \ f,g \in \LipX, k \leq f \oplus g \right\}.\]
The result now follows from the facts that $\min_{\pi} \pi(h) = - \max_{\pi} \pi(-h) = -\max_\pi \pi(k)$ and 
\[\sup\{\mu(f) + \nu(g) \ | \ f,g\in \LipX, f\oplus g\leq h\} = -\inf\{\mu(f) + \nu(g) \ | \ f,g\in \LipX, k\leq f\oplus g\}.\qedhere
\]
\end{proof}




By taking $h = \bar{d}$ in \cref{thm:MK-duality}, we get the existence of optimal couplings achieving the relative Wasserstein distance.
\begin{corollary}[Relative optimal transport]
  \label{cor:optimal_coupling}
    Assume that $(X,A)$ is boundedly compact.
	For any $\mu,\nu \in \pfRadon$, there exists $\pi^*\in \Pi(\mu,\nu)$ such that $W_1(\mu,\nu) = \pi^*(\bar{d})$.
\end{corollary}

\begin{theorem}[Relative Kantorovich-Rubinstein Duality]\label{thm:KR-duality}
    Assume that $(X,A)$ is boundedly compact.
	Let $\mu,\nu \in \pfRadon$. Then
	\[ W_1(\mu,\nu) =  \sup\{ \mu(f) - \nu(f) \ | \ f \in \LipOneX\}
	.\]
	Hence, viewing $\mu$ and $\nu$ as linear functionals on $\LipX$, we have $\norm{\mu-\nu}_{\textup{op}} =W_1(\mu,\nu)$.
\end{theorem}

\begin{proof}
By Monge-Kantorovich duality, Theorem \ref{thm:MK-duality},
\[W_1(\mu,\nu)  = \min_{\pi\in \Pi(\mu,\nu)} \pi(\bar{d}) = \sup\{ \mu(p) + \nu(q) \ | \ p,q \in \LipX, p \oplus q \leq \bar{d}\}.\]
By \cref{lem:lipschitz-dbar}, for $f\in \LipOneX$, we have $f\oplus (-f) \leq \bar{d}$.  
Hence
\[ \sup\{ \mu(f) - \nu(f) \ | \ f \in \LipOneX\} \leq \sup\{\mu(p) + \nu(q)\ | \ p,q \in \LipX, p \oplus q \leq \bar{d}\} = W_1(\mu,\nu).\]
To prove the reverse inequality, let $p,q \in \LipX$ be such that $p \oplus q \leq \bar{d}$.
Recall that $\bar{d}$ is a pseudometric and recall \cref{lem:lipschitz-dbar}.
We may now apply a standard argument to obtain a $1$-Lipschitz function from $p$ and $q$ (see \cite[Section 1.2]{villani2003topics} and \cite[Chapter 5]{Villani_2009}).
Define $p':X\to \R$ by $p'(x) = \inf_{y}(\bar{d}(x,y) - q(y))$ and then define $q':X\to \R$ by $q'(y) = \inf_{x}(\bar{d}(x,y)- p'(x))$.
Then $p'$ is 1-Lipschitz, $p' \oplus q' \leq \bar{d}$, $p \leq p'$, and $q\leq q'$. Moreover, it can be checked that $q' = -p'$.
Hence $\mu(p') - \nu(p') = \mu(p') + \nu(q') \geq \mu(p) + \nu(q)$, from which the desired inequality follows.
	
The last statement follows immediately from the definition of the operator norm.
\end{proof}

\begin{corollary}\label{cor:wasserstein_compact_op_norm} 
    Assume that 
    $(X,A)$ is boundedly compact.
	Let $\mu,\nu\in \pfRadon$. Then $W_1(\mu,\nu) = \sup\{\mu(f)- \nu(f) \ | \ f\in \LipcOneX\}$. Hence $W_1(\mu,\nu) = \norm{\mu-\nu}_{\textup{op}}$, where we view $\mu,\nu$ as positive linear functionals on $\LipcX$.
\end{corollary}

\begin{proof}
Clearly $\sup\{ \mu(f) - \nu(f) \ | \ f \in \LipcOneX\}\leq \sup\{ \mu(f) - \nu(f) \ | \ f \in \LipOneX\} = W_1(\mu,\nu)$. 

On the other hand, let $\eps > 0$ be given and let $f\in \LipOneX$ be such that $\mu(f)-\nu(f) > W_1(\mu,\nu) - \frac{\eps}{2}$. 
Then $f = f^+ - f^-$, where $f^+,f^- \in \LipOneXpos$.
By \cref{lem:pfRadon_exhausted_by_cpt_sets}, $\mu$ and $\nu$ are exhausted by compact sets.
So there exists a compact set $K \subset X\setminus A$ such that 
$\sup\{ \mu(g) \ | \ g \in \LipXpos, g|_K=0, g \leq f^+\} < \frac{\eps}{4}$,
$\sup\{ \mu(g) \ | \ g \in \LipXpos, g|_K=0, g \leq f^-\} < \frac{\eps}{4}$,
$\sup\{ \nu(g) \ | \ g \in \LipXpos, g|_K=0, g \leq f^+\} < \frac{\eps}{4}$, and
$\sup\{ \nu(g) \ | \ g \in \LipXpos, g|_K=0, g \leq f^-\} < \frac{\eps}{4}$.
Let $r^+ = \norm{f^+|_K}_{\infty}$ and
$r^- = \norm{f^-|_K}_{\infty}$.
By \cref{lem:compact-offset}, there exists a $\delta > 0$ such that $K^\delta$ is a compact subset of $X \setminus A$.
Choose $L>0$ such that $\frac{r^+}{L},\frac{r^-}{L} \leq \delta$.
Let $h^+ = (r^+ - L d_K) \meet f^+$ and
$h^- = (r^- - L d_K) \meet f^-$.
Then $h^\pm \in \LipcOneXpos$, $h^\pm \leq f^\pm$, and $h^\pm|_K = f^\pm|_K$.
Therefore $\mu(f^\pm) - \mu(h^\pm) = \mu(f^\pm-h^\pm) < \frac{\eps}{4}$,
$\nu(f^\pm) - \nu(h^\pm) = \nu(f^\pm-h^\pm) < \frac{\eps}{4}$,
$\mu(h^\pm) \leq \mu(f^\pm)$,
and $\nu(h^\pm) \leq \nu(f^\pm)$.
Let $h= h^+ - h^- \in \LipcOneX$.
Then $\mu(h) - \nu(h) = \mu(h^+) - \mu(h^-) - \nu(h^+) + \nu(h^-) 
> \mu(f^+) - \frac{\eps}{4} - \mu(f^-) - \nu(f^+) + \nu(f^-) - \frac{\eps}{4} 
= \mu(f) - \nu(f) - \frac{\eps}{2}
> W_1(\mu,\nu) - \eps$.
Hence $\sup\{ \mu(f) - \nu(f) \ | \ f \in \LipcOneX\} \geq W_1(\mu,\nu)$.


The last statement follows from the definition of the operator norm.
\end{proof}

Combining 
\cref{cor:representation-lip-non-positive,thm:KR-duality} 
gives our main result.

\begin{theorem}
  \label{thm:main_result}
    Assume that 
    $(X,A)$ is boundedly compact.
    For any order bounded, sequentially order continuous linear functional $T:\LipX\to \R$, there exists measures $\mu,\nu\in \pfRadon$ with $T(f) = \mu(f) - \nu(f)$ for all $f\in \LipX$. Moreover,
    $\|T\|_{\textup{op}} = W_1(\mu,\nu)$.
\end{theorem}

By combining Theorem \ref{thm:representation-proper-not-positive} with Corollary \ref{cor:wasserstein_compact_op_norm}, we get the following analogous statement for linear functionals on $\LipcX$.

\begin{theorem}
  \label{thm:main_result_compact}
  Assume that $(X,A)$ is boundedly compact.
  Let $T:\LipcX\to \R$ be an order bounded, sequentially order continuous linear functional.
  Suppose that $T$ and $|T|$ are bounded. Then there exists measures $\mu,\nu\in \pfRadon$ with $T(f) = \mu(f) - \nu(f)$ for all $f\in \LipcX$. Moreover, $\|T\|_{\textup{op}} = W_1(\mu,\nu)$.
\end{theorem}

\subsection{A norm for 1-finite real-valued Radon measures}
\label{sec:norm}

In this section, we show that our Wasserstein distance may be used to define a relative version of the Kantorovich-Rubinstein norm.
Let $(X,A)$ be a metric pair.
We assume that $(X,A)$ is boundedly compact.
Recall that 
$\pfRadon$ is 
an ideal in the zero-sum-free Riesz cone $\pRadon$ (\cref{prop:pRadon-convex-cone,prop:ideal})
and that
$W_1$ is a metric for $\pfRadon$ (\cref{thm:metric}).
We will show that $(\pfRadon,W_1)$ is a normed convex cone.


\begin{lemma}
    The metric $W_1$ on $\pfRadon$ is $\R^+$-homogeneous.
\end{lemma}

\begin{proof}
    Let $\mu,\nu \in \pfRadon$. Let $\alpha > 0$.
    There is a bijection of couplings $\Pi(\mu,\nu) \isomto \Pi(\alpha \mu, \alpha \nu)$ given by $\sigma \mapsto \alpha \sigma$.
    Furthermore the costs are related by $C(\alpha \sigma) = \alpha C(\sigma)$.
    Therefore $W_1(\alpha \mu, \alpha \nu) = \alpha W_1(\mu,\nu)$.
\end{proof}

\begin{lemma}
  The metric $W_1$ on $\pfRadon$ is translation invariant.
\end{lemma}

\begin{proof}
  Translation invariance of $W_1$ follows from Kantorovich-Rubinstein duality (\cref{thm:KR-duality}), since for $\lambda,\mu,\nu \in \pfRadon$,
 \begin{align*}
    W_1(\mu + \lambda,\nu + \lambda)
    &= \sup \{ \mu(f) + \lambda(f) - \nu(f) - \lambda(f) \ | \ f \in \LipOneX\}\\
    &= \sup \{ \mu(f) - \nu(f) \ | \ f \in \LipOneX\}
    = W_1(\mu,\nu). \qedhere
  \end{align*}
\end{proof}

Combining the previous two lemmas, we have the following.

\begin{proposition} \label{prop:pfRadon-normed-convex-cone}
    $(\pfRadon,W_1)$ is a normed convex cone.
\end{proposition}







Since the Grothendieck group of a normed cone $(C,\rho)$ is a normed vector space $(KC,\norm{-})$, with norm $\norm{x-y} = \rho(x,y)$, we have the following.

\begin{proposition} \label{prop:KR-norm}
    $\fRadon$ is a normed vector space with norm $\KRnorm{-}$ given by 
    \begin{equation*}
        \KRnorm{\mu-\nu} = W_1(\mu,\nu),
    \end{equation*}
    which we call the \emph{Kantorovich-Rubinstein norm}.
\end{proposition}

The following example shows that 
$\KRnorm{-}$ is not a lattice norm.

\begin{example}
  Consider the pointed metric space $\R$.
  Let $\mu = \delta_{2} + \delta_{8}$ and
  let $\nu = \delta_{2} - \delta_{3} + \delta_{8} - \delta_{9}$.
  Then $\abs{\mu} \leq \abs{\nu}$, but $\KRnorm{\mu} = W_1(\mu,0) = 10$
  and $\KRnorm{\nu} = W_1(\nu^+,\nu^-) = 2$.
\end{example}


Combining \cref{cor:representation-lip-non-positive,thm:KR-duality,prop:KR-norm}, we now have the following succinct summary of (relative) Kantorovich-Rubinstein duality.

\begin{theorem} \label{thm:KR-duality2}
  $\fRadon = \LipX_c^{\sim}$, and for $\mu \in \fRadon$, $\norm{\mu}_{\op} = \norm{\mu}_{\KR}$.
\end{theorem}

\subsection*{Acknowledgments}

 The first author was partially supported 
by the National Science Foundation (NSF) grant DMS-2324353;
by the Southeast Center for Mathematics and Biology, an NSF-Simons Research Center for Mathematics of Complex Biological Systems, under NSF Grant No. DMS-1764406 and Simons Foundation Grant No. 594594;
and by ARO Research Award W911NF1810307.


\printbibliography

\end{document}